\documentclass[envcountsame,envcountsect,a4paper,11pt]{svproc}
\usepackage{fullpage}

\usepackage{url}

\usepackage{amsmath,amssymb}
\usepackage{enumerate,hyperref}

\usepackage{graphicx}
\usepackage{makecell}
\usepackage{colortbl}

\usepackage{ifthen}
\usepackage{tikz}
\usetikzlibrary{backgrounds}
\usetikzlibrary{calc}
\usetikzlibrary{decorations.pathmorphing}
\usetikzlibrary{arrows.meta}

\tikzstyle{labeled}=[fill=white,draw=black,circle,inner sep=0, minimum size=4.8mm, thin]
\tikzstyle{simple}=[circle, draw=black, fill=white, inner sep=0, minimum size=2.2mm, thin]
\tikzstyle{tiny}=[circle, draw=black, fill=white, inner sep=0, minimum size=1.5mm, thin]
\tikzstyle{edge-label}=[circle, fill=white, inner sep=.12mm]

\spnewtheorem*{definitionn}{Definition}{\bf}{\rm}

\def\crg{\mathop{\mbox{\sl cr}}}

\begin{document}
\mainmatter              
\title{On 13-Crossing-Critical Graphs with Arbitrarily Large Degrees%
\,\thanks{This is a full and extended version of the paper published at the EuroComb~2021 conference as DOI 10.1007/978-3-030-83823-2\_9.}}
\titlerunning{On $13$-Crossing-Critical Graphs}
\author{Petr Hlin\v en\'y \and Michal Korbela}
\authorrunning{Petr Hlin\v en\'y and Michal Korbela}
\institute{Masaryk University, Faculty of Informatics,
        Botanick\'a 68a, Brno, Czech Republic\\ 
\email{hlineny@fi.muni.cz, kabell999@gmail.com}}
%
\maketitle              
\begin{abstract}
A recent result of Bokal et al.\ [Combinatorica, 2022] proved that {the exact minimum value} of $c$ such that 
$c$-crossing-critical graphs do \emph{not} have bounded maximum degree is~$c=13$.
The key to that result is an inductive construction of a family of $13$-crossing-critical graphs with many vertices of arbitrarily high degrees.
While the inductive part of the construction is rather easy, it all relies on the fact that a certain $17$-vertex base graph has the crossing
number $13$, which was originally verified only by a machine-readable  computer proof.
We provide a relatively short self-contained computer-free proof of the latter fact.
Furthermore, we subsequently generalize the critical construction in order to provide a definitive answer to a remaining open question
of this research area; we prove that for every $c\geq13$ and integers $d,q$, there exists a $c$-crossing-critical graph with more than $q$
vertices of {\em each} of the degrees $3,4,\ldots,d$.
\end{abstract}


\section{Introduction}

The {\em crossing number} $\crg(G)$ of a graph $G$ is the minimum number 
of (pairwise) edge crossings in a drawing of $G$ in the plane.
To resolve ambiguity, we consider drawings of graphs such that
no edge passes through another vertex and
no three edges intersect in a common point which is not their end.
A {\em crossing} is then an intersection point of two
edges that is not a vertex, and we always assume a finite number of crossings.

While for graphs with many (e.g., more than linear amount of) edges, it is not surprising to have a high crossing number,
graphs with relatively few edges (e.g., cubic ones) may have high crossing number if there is ``a lot of nonplanarity'' in them.
The latter is captured through the ``critical obstructions'' in which, informally, the crossing number drops down in every proper subgraph:
\begin{definition}
A graph $G$ is {\em$c$-crossing-critical} if $\crg(G)\ge c$, 
but for every edge~$e$ of $G$ we have $\crg(G-e)<c$.
\end{definition}
Note that our graphs, and in particular the crossing-critical graphs we are going to deal with,
are not required to be simple (they may contain parallel edges),
but one may always subdivide parallel edges without changing the crossing number.

There are two $1$-crossing-critical graphs up to subdivisions, $K_5$ and $K_{3,3}$,
but for every $c\geq2$ there exists an infinite number of $3$-connected $c$-crossing-critical graphs.
The latter fact nicely follows from a nowadays classical construction of Kochol~\cite{kochol87}, depicted in Figure~\ref{fig:crcrbasics}.
The same picture also shows another notable crossing-critical construction from~\cite{cit:pathcrit}.
These two constructions outline the two basic principles used in all
constructions of infinite $c$-crossing-critical families ever since;
using a suitably chosen ``planar belt'' which is joined at the ends twisted (as the M\"obius strip),
and similarly, a ``planar belt'' which is joined straight (as on a cylinder) and then crossed by a suitably attached additional edge(s).

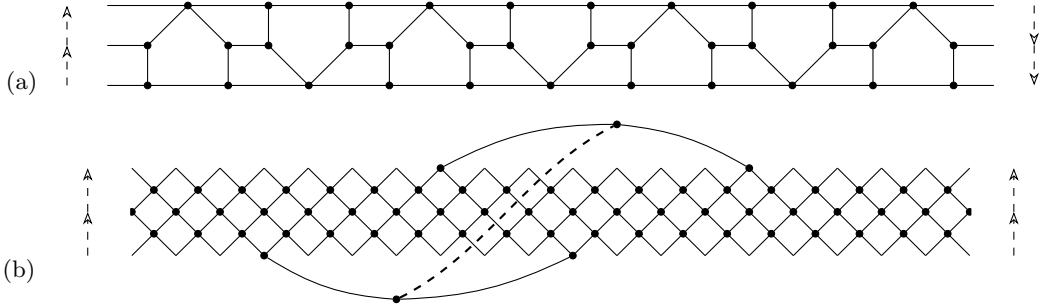
\begin{figure}[tb]
\centering
(a)\quad
\begin{tikzpicture}[scale=0.53]
\tikzstyle{every node}=[draw, shape=circle, minimum size=2.5pt,inner sep=0pt, fill=black]
\tikzstyle{every path}=[color=black]
\foreach \i in {0,...,3} {
\foreach \j in {1,-1} {
  \ifthenelse{\i<3 \OR \j=1}{
	\def\xx{6*\i+1.5-1.5*\j};
	\node at (\xx,1-\j) {};	\node at (\xx,1) {};
	\node at (\xx+1,1+\j) {};	\node at (\xx+2,1) {};
	\node at (\xx+2,1-\j) {};
 	\draw (\xx,1-\j) -- (\xx,1) -- (\xx+1,1+\j) -- (\xx+2,1) -- (\xx+2,1-\j);
 	\draw (\xx+2,1) -- (\xx+3,1);
  }{}
}}
\draw (-1,0) -- (21,0);	\draw (-1,1) -- (0,1);	\draw (-1,2) -- (21,2);
\foreach \i in {0,1} {
	\draw [dashed,-{Stealth[open]}] (24*\i-2,2*\i) -- (24*\i-2,1); 
	\draw [dashed,-{Stealth[open]}] (24*\i-2,1) -- (24*\i-2,2-2*\i); 
}
\end{tikzpicture}
\\\quad(b)
\begin{tikzpicture}[scale=0.58]
\useasboundingbox (0,-0.9) rectangle (23.2,4);
\tikzstyle{every node}=[draw, shape=circle, minimum size=2.5pt,inner sep=0pt, fill=black]
\tikzstyle{every path}=[color=black]
\foreach \i in {1,...,20} {
	\draw (\i,0) -- (\i+2,2);
	\draw (\i,2) -- (\i+2,0);
	\node at (\i+0.5,0.5) {};
	\node at (\i+1,1) {};
	\node at (\i+1.5,1.5) {};
}
\draw[color=white,fill=white] (0,0) rectangle (1.95,2);
\draw[color=white,fill=white] (21.05,0) rectangle (23,2);
\node (u) at (8,-1) {}; \node (u1) at (5,0) {}; \node (u2) at (12,0) {};
\node (v) at (13,3) {}; \node (v1) at (9,2) {}; \node (v2) at (16,2) {};
\draw[thick,dashed] (u) .. controls (10,0) and (11,2) .. (v);
\draw (u) to[bend right=12] (u2);
\draw (u) to[bend left=14] (u1);
\draw (v) to[bend left=12] (v2);
\draw (v) to[bend right=14] (v1);
\foreach \i in {1,22} {
	\draw [dashed,-{Stealth[open]}] (\i,0) -- (\i,1); 
	\draw [dashed,-{Stealth[open]}] (\i,1) -- (\i,2); 
}
\end{tikzpicture}
\caption{An illustration of the two basic approaches to crossing-critical graph constructions.
(a) The classical construction of $2$-crossing-critical graphs by Kochol 
in which the ends of the planar belt are joined twisted. 
(b)~A~construction of $c$-crossing-critical graphs ($c\geq3$, here $c=4$) 
by Hlin\v en\'y in which the
ends of the middle planar belt are joined straight, not twisted.}
\label{fig:crcrbasics}
\end{figure}

A long-sought asymptotic characterization of $c$-crossing-critical families
(for every fixed~$c$) has been proved only recently by Dvo\v r\'ak et al.~\cite{dvorakHlinenyMohar18},
confirming prior experience with the two aforementioned principles of crossing-critical constructions.
Namely, there are finitely many basic $c$-crossing-critical graphs for every $c\geq2$, and all other $c$-crossing-critical graphs
can be constructed iteratively from the basic graphs roughly as follows.
Take one of the graphs and find a suitable (long and thin) ``planar belt'' within it, which is attached to the rest of the graph
only at the ends, straight or twisted.
Then ``prolong'' the internal part of this belt by duplicating suitable well-defined repeated substructures within it.

Constructions as sketched in Figure~\ref{fig:crcrbasics}, in particular, naturally have an upper bound on the maximum degree in terms of~$c$.
However, the characterization given in~\cite{dvorakHlinenyMohar18} leaves room for a rather obscure third possibility of a construction
in which, informally saying, a part of the twisted planar belt boundary
(as in Figure~\ref{fig:crcrbasics}(a)) ``collapses'' into a single vertex of high degree.
The existence of such a construction for high values of $c$ has been implicitly confirmed already by Dvo\v r\'ak and Mohar~\cite{cit:dvorakmohar},
that time disproving an unpublished opposite conjecture of Richter,
but no explicit examples were known until a more recent exhaustive work of Bokal et al.~\cite{DBLP:journals/combinatorica/BokalDHLMW22}:


\smallskip
\begin{theorem}[Bokal, Dvo\v{r}{\'a}k, Hlin\v{e}n{\'y}, Lea{\~n}os, Mohar and Wiedera~\cite{DBLP:journals/combinatorica/BokalDHLMW22}]~
\label{thm:main13}
\\a) For each $1\leq c\leq12$, there exists a constant $D_c$ such that every $c$-crossing-critical graph has vertex degrees at most~$D_c$.
\\b) For each $c\geq13$ and every integers $q,d$,
one can construct a $c$-crossing-critical graph which has at least $q$ vertices of degree at least~$d$.
\end{theorem}

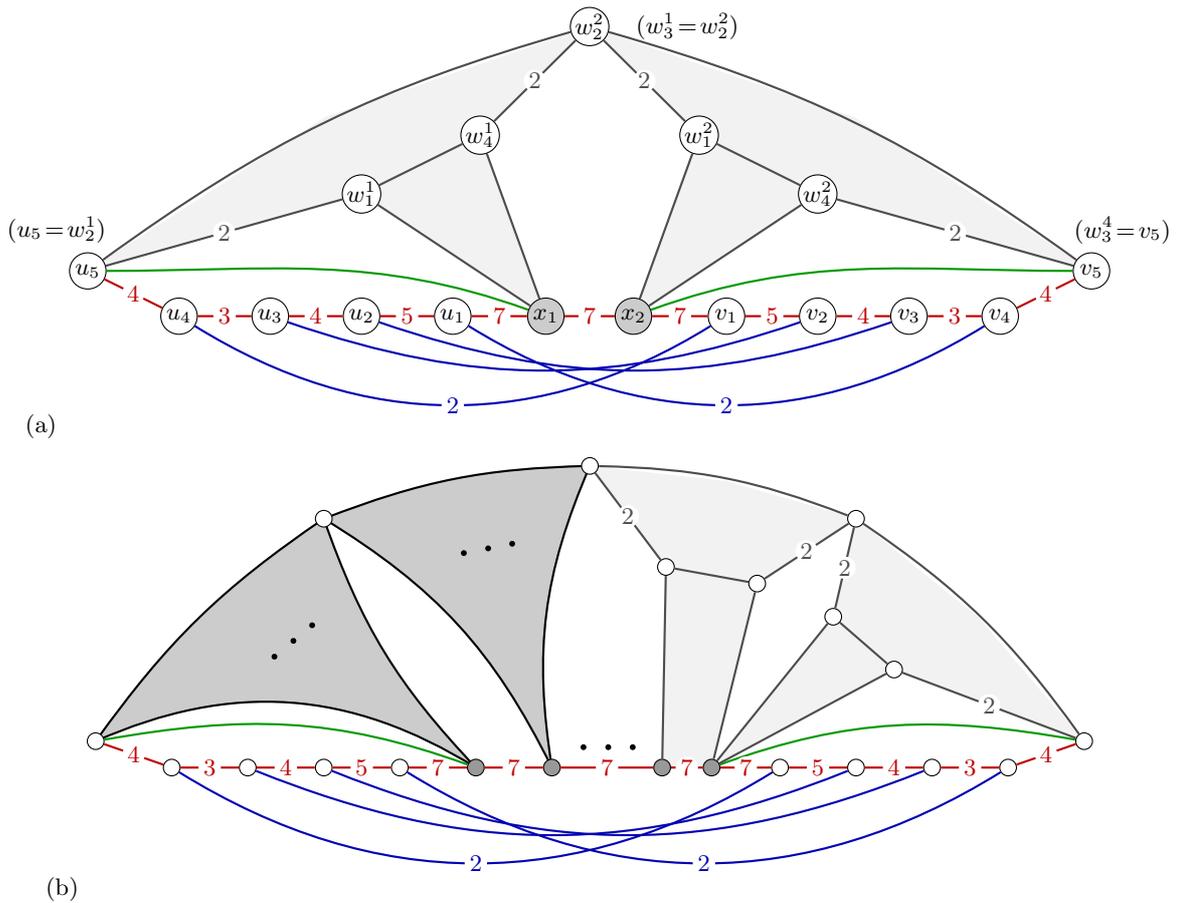
\begin{figure}
 \centering
(a)
 \begin{tikzpicture}[scale=1.2]\footnotesize
  \tikzset{every node/.style={labeled}}
  \tikzset{every path/.style={thick}}
  \node[fill=black!20!white] (x) at (-0.98,0) {$x_1$};
  \node[fill=black!20!white] (xx) at (-0.02,0) {$x_2$};
  \node (u1) at (-2,0) {$u_1$};
  \node (u2) at (-3,0) {$u_2$};
  \node (u3) at (-4,0) {$u_3$};
  \node (u4) at (-5,0) {$u_4$};
  \node[label=above:{{\hspace*{-6ex}($u_5\!=\!w_2^1$)}}] (u5) at (-6,0.5) {$u_5$};
  \node (v1) at (1,0) {$v_1$};
  \node (v2) at (2,0) {$v_2$};
  \node (v3) at (3,0) {$v_3$};
  \node (v4) at (4,0) {$v_4$};
  \node[label=above:{{($w_3^4\!=\!v_5$)\hspace*{-6ex}}}] (v5) at (5,0.5) {$v_5$};
  \node (h2) at (-3,1.35) {$w^{1}_1$};
  \node (h3) at (-1.7,2) {$w^{1}_4$};
  \node[label=right:{{\quad($w_3^{1}\!=\!w^2_2$)}}] (k1) at (-0.5,3.2) {$w^2_2$};
  \node (k2) at (0.7,2) {$w^{2}_1$};
  \node (k3) at (2,1.35) {$w^{2}_4$};
  \tikzset{every node/.style={edge-label}}
  \draw[green!60!black] (x) to[in=0, out=160] (u5);
  \draw[green!60!black] (xx) to[out=20, in=180] (v5);
  \draw[red!75!black] (u5) -- node {$4$} (u4) -- node {$3$} (u3)
   -- node {$4$} (u2) -- node {$5$} (u1) -- node {$7$} (x) -- node {$7$} (xx)
   -- node {$7$} (v1) -- node {$5$} (v2) -- node {$4$} (v3) -- node {$3$} (v4) -- node {$4$} (v5);
  \draw[blue!70!black]
   (u2) edge[bend right=19] (v3)
   (u3) edge[bend right=19] (v2)
   (u1) edge[bend right=32] node {2} (v4)
   (u4) edge[bend right=32] node {2} (v1)
  ;
  \draw[black!70!white] (u5) edge[thick] node {2} (h2) (h2) -- (h3) (h3) edge[thick] node {2} (k1)
     (k1) edge[bend right=10] (u5)
   (k1) edge[thick] node {2} (k2) (k2) -- (k3) (k3) edge[thick] node {2} (v5)
     (v5) edge[bend right=10] (k1)
   (h2) -- (x) -- (h3)  (k2) -- (xx) -- (k3)
  ;
  \begin{scope}[on background layer]
    \fill[color=black!5!white] (x.center) -- (h2.center) -- (u5.center)
      -- (-3.4,2.2) -- (k1.center) -- (h3.center) -- cycle;
    \fill[color=black!5!white] (xx.center) -- (k3.center) -- (v5.center)
      -- (2.4,2.2) -- (k1.center) -- (k2.center) -- cycle;
  \end{scope}
 \end{tikzpicture}

\medskip(b)
 \begin{tikzpicture}[scale=1]\footnotesize
  \tikzset{every node/.style={simple}}
  \tikzset{every path/.style={thick}}
  \node[fill=black!40!white] (x) at (-1,0) {};
  \node[fill=black!40!white] (x1) at (0,0) {};
  \node[fill=black!40!white] (xx1) at (1.45,0) {};
  \node[fill=black!40!white] (xx) at (2.1,0) {};
  \node (u1) at (-2,0) {};
  \node (u2) at (-3,0) {};
  \node (u3) at (-4,0) {};
  \node (u4) at (-5,0) {};
  \node (u5) at (-6,0.35) {};
  \node (v1) at (3,0) {};
  \node (v2) at (4,0) {};
  \node (v3) at (5,0) {};
  \node (v4) at (6,0) {};
  \node (v5) at (7,0.35) {};
  \node (g1) at (-3,3.3) {};
  \node (h1) at (0.5,4) {};
  \node (h2) at (1.5,2.66) {};
  \node (h3) at (2.7,2.44) {};
  \node (k1) at (4,3.3) {};
  \node (k2) at (3.7,2) {};
  \node (k3) at (4.5,1.3) {};
  \tikzset{every node/.style={edge-label}}
  \draw[green!60!black] (x) to[in=10, out=160] (u5);
  \draw[green!60!black] (xx) to[out=20, in=170] (v5);
  \draw[red!75!black] (u5) -- node {$4$} (u4) -- node {$3$} (u3) -- node {$4$}
   (u2) -- node {$5$} (u1) -- node {$7$} (x) -- node {$7$} (x1) -- node {$7$} (xx1) -- node {$7$} (xx) -- node {$7$} (v1) -- node {$5$}
   (v2) -- node {$4$} (v3) -- node {$3$} (v4) -- node {$4$} (v5);
  \draw[blue!70!black]
   (u2) edge[bend right=22] (v3)
   (u3) edge[bend right=22] (v2)
   (u1) edge[bend right=32] node {2} (v4)
   (u4) edge[bend right=32] node {2} (v1)
  ;
  \draw[black!70!white] (h1) edge[thick] node {2} (h2) (h2) -- (h3) (h3) edge[thick] node {2} (k1) (k1) edge[bend right=10] (h1)
   (k1) edge[thick] node {2} (k2) (k2) -- (k3) (k3) edge[thick] node {2} (v5) (v5) edge[bend right=10] (k1) (h2) -- (xx1)
   (xx) -- (h3) (k2) -- (xx) -- (k3)
  ;
  \begin{scope}[on background layer]
   \draw[fill=black!20!white] (x.center) to[bend right=28] (u5.center) to[bend left=10] (g1.center) to[bend right=12] cycle;
   \draw[fill=black!20!white] (x1.center) to[bend right=16] (g1.center) to[bend left=10] (h1.center) to[bend right=16] cycle;
   \fill[color=black!5!white] (xx.center) -- (xx1.center) -- (h2.center) -- (h1.center) -- (2.3,3.8) -- (k1.center) -- (h3.center) -- cycle;
   \fill[color=black!5!white] (xx.center) -- (k3.center) -- (v5.center) -- (5.6,2) -- (k1.center) -- (k2.center) -- cycle;
  \end{scope}
  \node[rotate=40, draw=none, fill=none] (label) at (153:3.75) {\huge$\cdots$};
  \node[rotate=11, draw=none, fill=none] (label) at (105:3) {\huge$\cdots$};
  \node[draw=none, fill=none] (label) at (0.8,0.25) {\huge$\cdots$};
 \end{tikzpicture}

 \caption{The inductive construction of $13$-crossing-critical graphs from
Theorem~\ref{thm:main13} (note that all $13$ depicted crossings are only between the blue edges).
The edge labels in the picture represent the number of parallel edges
between their end vertices (e.g., there are $7$ parallel edges between $x_1$ and~$x_2$).
Figure (a) defines the \emph{base graph $G_{13}$} of the construction, and (b) outlines
the general construction which arbitrarily ``duplicates'' the two wedge-shaped gray subgraphs of $G_{13}$ and the gray vertices $x_1,x_2$,
and possibly also ``splits'' the tips of the wedge-shaped gray subgraphs.
See further Section~\ref{sec:improved}.}
 \label{fig:main13}
\end{figure}

The critical construction of Theorem~\ref{thm:main13}(b) is outlined for~$c=13$ in Figure~\ref{fig:main13}.
Figure~\ref{fig:main13}(a) defines the $17$-vertex $13$-crossing-critical (multi)graph $G_{13}$ 
which is the base graph of the full inductive construction.
One can see in~\cite{DBLP:journals/combinatorica/BokalDHLMW22} that the proof of Theorem~\ref{thm:main13}(b) follows straightforwardly 
(using induction) from the fact that $\crg(G_{13})\geq13$.
However, for the claim that $\crg(G_{13})\geq13$, only a machine-readable computer proof is provided in~\cite{DBLP:journals/combinatorica/BokalDHLMW22};
the proof is based on an ILP branch-and-cut-and-price routine~\cite{chimaniWiedera16} with about a thousand cases of up to hundreds of constraints each.
Our first goal is to provide a much simpler handwritten proof:

\begin{theorem}[a computer-free alternative to Theorem~\ref{thm:main13}(b)]
\label{thm:main13alt}
$\crg(G_{13})\geq13$.
\end{theorem}

Secondly, we further generalize the construction such that we can prove the following strengthening of Theorem~\ref{thm:main13}(b):
\begin{theorem}\label{thm:main13impro}
For each $c\geq13$ and every integers $q,d$, one can construct a $3$-connected $c$-crossing-critical graph 
which has at least $q$ vertices of each of the degrees $3$, $4$, $\ldots$,~$d$.
\end{theorem}

We remark that in both Theorem~\ref{thm:main13impro} and Theorem~\ref{thm:main13}, the constructed graphs
can be either chosen to be $3$-connected, or $2$-connected and simple (via a subdivision of the former graphs).
We, however, do not know whether the conditions of $3$-connectivity and simplicity can be combined together.

\section{Theorem~\ref{thm:main13alt} without Assistance of a Computer}

We divide the self-contained proof of Theorem~\ref{thm:main13alt} into two steps.
We call {\em red} the edges of $G_{13}$ (Fig.~\ref{fig:main13}(a)) which form the path with multiple edges on
$(u_5,u_4,u_3,u_2,u_1,x_1,x_2,v_1,v_2,v_3,v_4,v_5)$,
and we call {\em blue} the edges with one end $u_i$ and other end $v_j$ where $i,j\in\{1,2,3,4\}$.
\begin{enumerate}[(1)]
\item We will first show that there is an {\em optimal drawing} (i.e., one minimizing the
number of crossings) of $G_{13}$ such that no red edge crosses a red or a
blue edge. Note that blue--blue crossings are still allowed (and likely to occur).
\item Secondly, while considering drawings as in the first point,
we will refine our analysis by counting only selected crossings
(roughly, those involving a blue edge),
and prove at least $13$ of them, or at least $12$ with the remaining drawing
still being non-planar.
\end{enumerate}

\subsection{Restricting optimal drawings of $G_{13}$}

Before proceeding further, we need some basic facts about the crossing number.

\begin{proposition}[folklore]\label{pro:folklorecr}
a) If $D$ is an optimal drawing of a graph $G$, then no two edges cross more
than once and edges sharing a common end do not cross at all in~$D$.
\\b) If $e$ and $f$ are parallel edges in $G$ (i.e., $e,f$ have the same
end vertices), then there is an optimal drawing of $G$ in which $e$ and $f$
are drawn ``closely together'', meaning that they cross the same other edges
in the same order.
\end{proposition}
\noindent
In view of Proposition~\ref{pro:folklorecr}(b), we adopt the following view
of multiple edges:
If the vertices $u$ and $v$ are joined by $p$ parallel edges, we view all $p$ of them as one edge $f$ of {\em weight~$p$}
(cf.~Figure~\ref{fig:main13} in which the weights $>1$ are shown as the edge labels).
If (multiple) edges $f$ and $g$ of weights $p$ and $q$ cross each other,
then their crossing naturally contributes the amount of $p\cdot q$ 
to the total number of crossings.
This folklore observation greatly simplifies the analysis of our
graph $G_{13}$ which contains many multiple edges.

We now, with help of the previous observations, finish the first step:
\begin{lemma}\label{lem:redbluecr}
There exists an optimal drawing of the graph $G_{13}$ in which no red edge
crosses a red or a blue edge, or~$\crg(G_{13})\geq13$.
\end{lemma}

Note that the drawing in Fig.~\ref{fig:main13}(a) is of the kind anticipated
by Lemma~\ref{lem:redbluecr} (and there are also other drawings of $G_{13}$ with 
$13$ crossings of this kind, which can be very different from that of
Fig.~\ref{fig:main13}(a)), but we do not known yet at this stage whether they are optimal.

\begin{proof}
Let $D$ be an optimal drawing of $G_{13}$. 
If two red edges $e$ and~$e'$ of $G_{13}$, viewed as weighted edges by Proposition~\ref{pro:folklorecr}(b), cross in $D$, 
then the weights of $e$ and $e'$ are $3$ and $3$, or $3$ and~$4$ (see Figure~\ref{fig:main13}),
and this crossing alone accounts for at least $9$ crossings in~$D$.
Since the edge $u_4u_3$ cannot cross $u_3u_2$ or $u_5u_4$ by Proposition~\ref{pro:folklorecr}(a),
and by symmetry, we have that $e\in\{u_5u_4,u_4u_3,u_3u_2\}$ and $e'\in\{v_5v_4,v_4v_3,v_3v_2\}$ (and not both of weight~$4$).
Consi\-der two edge-disjoint simple (weight $1$) cycles $C_1=(u_5,u_4,u_3,u_2,u_1,x_1)$ and $C_2=(u_5,u_4,u_3,u_2,u_1,x_1,w_1^1)$,
and symmetric $C_1'=(v_5,v_4,v_3,v_2,v_1,x_1)$, $C_2'=(v_5,v_4,v_3,v_2,v_1,x_1,w_1^1)$.
Since $C_1,C_2$ transversely cross $C_1',C_2'$, they must cross a second time
by the Jordan curve theorem, giving additional $2\cdot2=4$ crossings.
Hence $D$ had at least $9+4=13$ crossings, and~$\crg(G_{13})\geq13$.

It remains to get rid of possible red--blue crossings in~$D$.
Assume first that the red edge $x_1x_2$ (weight~$7$) is crossed by a blue edge~$e$.
If $e$ is of weight $2$, then $D$ has~$14$ cros\-sings.
Hence, up to symmetry, $e=u_3v_2$.
We redraw $e$ closely along the path $P=(u_3,u_4,v_1,v_2)$,
saving $7$ crossings on $x_1x_2$, and newly crossing $u_4v_1$ (possibly)
and~$p$~edges which cross $P$ in $D$.
If $p\geq3$, there are already $7+p\cdot2\geq13$ crossings~in~$D$.
Otherwise, redrawing $e$ makes only $p+2\leq4$ new crossings, contradicting optimality of~$D$.

In the rest, we iteratively remove unwanted red--blue crossings by redrawing 
the involved blue edge, while not increasing the total number of crossings. 
Since we always remove some red--blue crossing,
this will eventually lead to a desired optimal drawing of~$G_{13}$.

Assume in $D$ that a red edge $v_3v_4$ (or, symmetrically $u_4u_3$) is crossed by
a blue edge~$e$, and this is the blue crossing on $v_3v_4$ closest to~$v_3$.
In this case, $e$ has one end $w\in\{v_1,v_2\}$.
Instead of crossing $v_3v_4$ (saving $3$ crossings),
we redraw part of $e$ closely along the path $(v_3,v_2,v_1)$ to the end~$w$,
while possibly crossing blue $v_3u_2$, $v_2u_3$ (if $w=v_1$), 
and other $r$ edges which in $D$ cross red $v_3v_2$ or $v_2v_1$.
If $r\geq3$, we already had $r\cdot4+3\geq15$ crossings in~$D$.
If $r\geq2$ and $e=v_1u_4$ (weight~$2$), we had
$r\cdot4+3\cdot2\geq14$ crossings~in~$D$.~%
Otherwise, redrawn $e$ crosses at most $3$ new edges, and so we have
no more crossings than in~$D$.

Finally, if the red edge $v_4v_5$ is crossed by blue~$e$ in~$D$, 
then we ``pull'' the end $v_5$ along $v_4v_5$ towards $v_4$ across $e$.
This redrawing replaces the crossing of $e$ with $v_4v_5$ of weight~$4$
by $4$ crossings with the non-red edges of~$v_5$,
so again no more crossings than in~$D$.
If red $v_3v_2$ ($v_2v_1$, or $v_1x_2$) is crossed by blue $e$ in $D$,
we similarly ``pull'' along the red edges $v_3$ towards $v_2$
($v_2,v_3$ towards $v_1$, or $v_1,v_2,v_3$ towards $x_2$).
This replaces the original crossing of $e$ with red by crossings of $e$
with $v_3v_4$ and $v_3u_2$ (plus $v_2u_3$ or plus $v_2u_3,v_1u_4$),
but the total number of crossings stays the same as in~$D$.
Then we redraw the crossing of $e$ and $v_3v_4$ as above.
\qed\end{proof}

\subsection{Counting selected crossings in a drawing of $G_{13}$}

In the second step, we introduce two additional sorts of edges of~$G_{13}$.
The edges $u_5x_1$ and $v_5x_2$ are called {\em green},
and all remaining edges of~$G_{13}$ (i.e., not blue or red or green) are declared {\em gray}.
Let $G_0$ denote the subgraph of $G_{13}$ formed by all red and gray edges
and the incident vertices.
Let $R$ denote the (multi)path of all red edges.


\begin{lemma}\label{lem:redgrayplanar}
Let $D$ be an optimal drawing of the graph $G_{13}$ as claimed by Lemma~\ref{lem:redbluecr}.
If the subdrawing of $G_0$ within $D$ is planar, then $D$ has at least $13$ crossings.
\end{lemma}

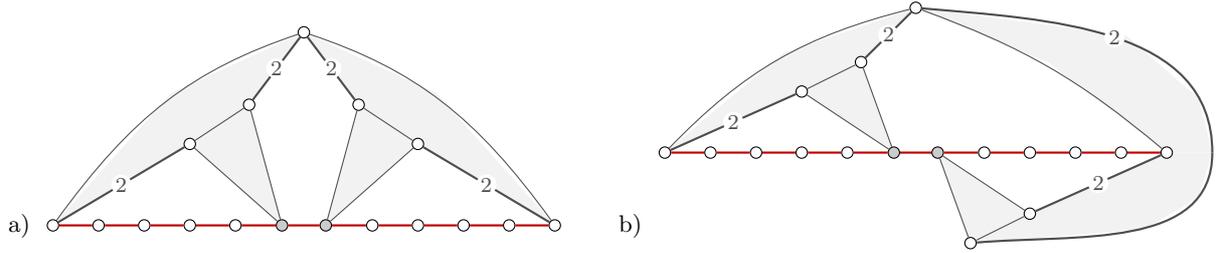
\begin{figure}
 \centering
 a)~
 \begin{tikzpicture}[xscale=0.6,yscale=0.8]\scriptsize
  \tikzset{every node/.style={tiny}}
  \node[fill=black!20!white] (x) at (-0.98,0) {};
  \node[fill=black!20!white] (xx) at (-0.02,0) {};
  \node (u1) at (-2,0) {};
  \node (u2) at (-3,0) {};
  \node (u3) at (-4,0) {};
  \node (u4) at (-5,0) {};
  \node (u5) at (-6,0) {};
  \node (v1) at (1,0) {};
  \node (v2) at (2,0) {};
  \node (v3) at (3,0) {};
  \node (v4) at (4,0) {};
  \node (v5) at (5,0) {};
  \node (h2) at (-3,1.35) {};
  \node (h3) at (-1.7,2) {};
  \node (k1) at (-0.5,3.2) {};
  \node (k2) at (0.7,2) {};
  \node (k3) at (2,1.35) {};
  \tikzset{every node/.style={edge-label}}
  \draw[red!75!black,thick] (u5) -- (u4) -- (u3)
   -- (u2) -- (u1) -- (x) -- (xx)
   -- (v1) -- (v2) -- (v3) -- (v4) -- (v5);
  \draw[black!70!white] (u5) edge[thick] node {2} (h2) (h2) -- (h3) (h3) edge[thick] node {2} (k1)
     (k1) edge[bend right=16] (u5)
   (k1) edge[thick] node {2} (k2) (k2) -- (k3) (k3) edge[thick] node {2} (v5)
     (v5) edge[bend right=16] (k1)
   (h2) -- (x) -- (h3)  (k2) -- (xx) -- (k3)
  ;
  \begin{scope}[on background layer]
    \fill[color=black!5!white] (x.center) -- (h2.center) -- (u5.center)
      -- (-3.4,2.2) -- (k1.center) -- (h3.center) -- cycle;
    \fill[color=black!5!white] (xx.center) -- (k3.center) -- (v5.center)
      -- (2.4,2.2) -- (k1.center) -- (k2.center) -- cycle;
  \end{scope}
 \end{tikzpicture}
\qquad
 b)%
 \begin{tikzpicture}[xscale=0.6,yscale=0.6]\scriptsize
  \path[use as bounding box] (-6.5,-1.75) rectangle (6,3.3);
  \tikzset{every node/.style={tiny}}
  \node[fill=black!20!white] (x) at (-0.98,0) {};
  \node[fill=black!20!white] (xx) at (-0.02,0) {};
  \node (u1) at (-2,0) {};
  \node (u2) at (-3,0) {};
  \node (u3) at (-4,0) {};
  \node (u4) at (-5,0) {};
  \node (u5) at (-6,0) {};
  \node (v1) at (1,0) {};
  \node (v2) at (2,0) {};
  \node (v3) at (3,0) {};
  \node (v4) at (4,0) {};
  \node (v5) at (5,0) {};
  \node (h2) at (-3,1.35) {};
  \node (h3) at (-1.7,2) {};
  \node (k1) at (-0.5,3.2) {};
  \node (k2) at (0.7,-2) {};
  \node (k3) at (2,-1.35) {};
  \tikzset{every node/.style={edge-label}}
  \draw[red!75!black,thick] (u5) -- (u4) -- (u3)
   -- (u2) -- (u1) -- (x) -- (xx)
   -- (v1) -- (v2) -- (v3) -- (v4) -- (v5);
  \draw[black!70!white] (u5) edge[thick] node {2} (h2) (h2) -- (h3) (h3) edge[thick] node {2} (k1)
     (k1) edge[bend right=16] (u5)
   (k1) edge[thick,out=-5,in=90] node {2} (6,0) (6,0)
edge[thick,out=-90,in=5] (k2)
     (k2) -- (k3) (k3) edge[thick] node {2} (v5) (v5) edge[bend right=10] (k1)
   (h2) -- (x) -- (h3)  (k2) -- (xx) -- (k3)
  ;
  \begin{scope}[on background layer]
    \fill[color=black!5!white] (x.center) -- (h2.center) -- (u5.center)
      -- (-3.4,2.2) -- (k1.center) -- (h3.center) -- cycle;
    \fill[color=black!5!white] (xx.center) -- (k3.center) -- (v5.center)
      -- (6,0) -- (5.7,-1) -- (4.5,-1.7) -- (k2.center) -- cycle;
    \fill[color=black!5!white] (v5.center)
      -- (6,0) -- (5.7,1.2) -- (4.1,2.6) -- (k1.center) -- (2.4,2) -- cycle;
  \end{scope}
 \end{tikzpicture}

 \caption{Schematically, the two non-equivalent planar drawings of the subgraph
   $G_0$ (red and gray) of~$G_{13}$. This is used in the proof of Lemma~\ref{lem:redgrayplanar}.}
 \label{fig:tworedgray}
\end{figure}

\begin{proof}
There are only two non-equivalent planar drawings of $G_0$, as in Fig.~\ref{fig:tworedgray}.
We picture them with the red path $R$ drawn as a horizontal line (and we recall that $R$ is not crossed by blue edges).
We call a blue edge of $G_{13}$ {\em bottom} if it is attached to $R$ from below
at both ends, and {\em top} if attached from above at both ends.
A blue edge is {\em switching} if it is neither top nor bottom.
Note that we have only crossings involving a green edge, or
crossings of a blue edge with a blue~or~gray~edge.

Consider the drawing of $G_0$ in Fig.~\ref{fig:tworedgray}(a). By the Jordan curve theorem, we first deduce:
\begin{enumerate}[(I)]
\item If a blue edge $e$ is bottom (top), and a blue edge $e'\not=e$ attaches to $R$ from below (from above)
at its end which is between the ends of $e$ on~$R$, then $e$ and $e'$ cross.
In particular, two bottom (two top) blue edges always cross each other.
\item A top blue edge crosses at least $4$ gray edges, and a switching one at least~$3$ gray edges.
\item If there is weight $k$ of top blue edges and weight $\ell$ of bottom blue edges, 
then each of the two green edges has at least $\min(k,\ell)\leq3$ crossings with blue or red edges.
\item If there is weight $k$ of switching blue edges and weight $\ell$ of bottom blue edges, 
then the two green edges together have at least $\min(k,\ell)$ crossings with blue or red edges.
\end{enumerate}

We refer to Fig.~\ref{fig:main13}(a).
While a proof of (I) is straightforward, we provide the details for the claims (II) to (IV).

In (II); for a top edge $f$ (which cannot cross~$R$) we argue that $f$ must
cross the four edge-disjoint gray paths $(u_5,w_1^1,x_1)$,
$(u_5,w_1^1,w_4^1,x_1)$, $(v_5,w_4^2,x_2)$, and $(v_5,w_4^2,w_1^2,x_2)$.
For a switching edge $f$, which up to symmetry attaches to $u_i$ from above,
we argue that $f$ must cross the three edge-disjoint gray paths $(u_5,w_1^1,x_1)$,                  
$(u_5,w_1^1,w_4^1,x_1)$ and $(u_5,w_2^2,w_1^2,x_2)$.

In (III); considering (possibly parallel) $k$ top and $\ell$ bottom blue edges, 
the drawings of each cycle consisting of one top, one bottom, and between two
and six interconnecting red edges separates $x_1$ from $u_5$ and $x_2$ from $v_5$, 
and so there must be an additional crossing on each of the green edges $u_5x_1$ and $x_2v_5$.
There are $\min(k,\ell)\leq3$ such edge-disjoint cycles, and hence the
claimed minimum number of crossings with each green edge.

In (IV); we define the $\min(k,\ell)$ edge-disjoint cycles analogously as in (III) (but this time with possible self-crossings); 
now, each of them is guaranteed to separate only one of $x_1$ from $u_5$ or $x_2$ from $v_5$,
and so we get a lesser conclusion of (IV) the same way.

\smallskip
Next, we summarize the case of $G_0$ as in Fig.~\ref{fig:tworedgray}(a) according to the number of bottom edges.
If no blue edge is bottom, then we have at least $6\cdot3=18$ crossings by (II).
If the combined weight of the bottom edges is at most $2$,
then we have at least $4\cdot3=12$ crossings by (II), plus a crossing by (III),(IV).
If the weight of the bottom edges equals $3$, then we have $2$ crossings by (I), 
at least $3\cdot3=9$ crossings by (II), and at least $2$ crossings by (III), summing to~$\geq13$.
If the weight of the bottom edges equals $4$, then we have at least $5$ crossings by (I), 
at least $2\cdot3=6$ crossings by (II), and at least $2$ crossings by (III), summing again to at least~$13$.

We are left with the cases in which the weight of the bottom edges is $6$ or~$5$.
If all blue edges are bottom, they pairwise give desired ${6\choose2}-2=13$ crossings by~(I).
Otherwise, one of $u_3v_2,u_2v_3$, say $e$, is top or switching and all other blue edges are bottom.
If $e$ is top, then we get $8$ crossings by (I), $4$ crossings by (II) and $2$ crossings by (III), together $8+4+2=14$ crossings.
If $e$ is switching, then we get at least $8+2=10$ crossings by (I), $3$ crossings by (II) and $1$ crossing by (III), together $10+3+1=14$ again.

\medskip
To finish, we consider the other drawing of $G_0$ in Fig.~\ref{fig:tworedgray}(b).
Now a top or bottom blue edge must cross at least $2$ gray edges, and a switching blue edge at least $3$ gray edges.
Hence we get desired $13$ crossings, except when no blue edge is switching 
and we have the minimum~of~$6\cdot2=12$ blue-gray crossings.
Though, in the latter case we get another blue crossing as~in~(I).
\qed\end{proof}

We now focus on the following selected crossings in a drawing of $G_{13}$:
a {\em refined crossing} is one in which a blue edge crosses a gray or green edge, or two blue edges cross each other.
This will help the remaining analysis.

\begin{lemma}\label{lem:redgraynonplanar}
Let $D$ be a drawing of the graph $G_{13}$ as claimed by Lemma~\ref{lem:redbluecr}.
If $D$ has less than $13$ crossings and no red edge is crossed in $D$,
then $D$ contains $12$ refined crossings.
\end{lemma}

\begin{proof}
We may again picture the red path $R$ as a straight horizontal line,
and use the terms {\em top/bottom/switching}\/ for blue edges as in the proof
of Lemma~\ref{lem:redgrayplanar}.
Then we consider the following six pairwise edge-disjoint gray and green paths:
$P_1$ (of length~$2$) and $P_2$ (of length~$6$) join $u_5$ to $v_5$ via $w_2^2$,~
$Q_1$ and $Q_1'$ are formed by the edges $u_5x_1$ and $v_5x_2$,
and $Q_2,Q_2'$ (of length~$2$) join $u_5$ to $x_1$ via $w_1^1$
and $v_5$ to $x_2$ via $w_4^2$.
Using the Jordan curve theorem (cf.~Fig.~\ref{fig:main13}), we deduce:
\begin{enumerate}[(I$'$)]
\item The same claim for blue--blue crossing as (I) in the proof of Lemma~\ref{lem:redgrayplanar}
applies here.
\item If a blue edge $e$ is switching, then $e$ adds refined crossings
on both $P_1$ and~$P_2$.
\item If the sum of weights of the blue edges attached to 
$u_1,u_2,u_3,u_4$ on $R$ from below is~$k$ (and so $6-k$ from above), 
then each of $Q_1,Q_2$ carries at least $\min(k,6-k)$ unique
crossings with blue edges in~$D$.
This symmetrically applies to $Q_1',Q_2'$.
\end{enumerate}
All three claims are argued very similarly to the proof of Lemma \ref{lem:redgrayplanar}.

It now remains to routinely examine all combinations of the blue edges being top/ bottom/switching, 
and in each one use the claims (I$'$) to (III$'$) to argue that $D$ contains at least $12$ pairwise distinct refined crossings.
Unfortunately, unlike in the previous proof, we now have quite many dissimilar tight or nearly tight cases.
Although each of them is quite easy, it is better to summarize all cases with their guaranteed lower bounds in one Table~\ref{tbl2} for clarity.
Within this summary, we denote by $E_1$ and $E_2$ the sets of the top and bottom blue edges, in order
such that the combined weight of $E_1$ is at least the combined weight of~$E_2$ -- since our arguments are fully symmetric in top/bottom edges.
Let $E_s$ the set of remaining switching blue edges.
%

\begin{table}[t]
\centering\normalsize
\begin{tabular}{| c | c | c | c | c | c || c || }
\hline
$E_1$ & $E_2$ & $E_s$ & {\scriptsize \makecell{(I$'$) pairwise\\ crossings \\in $E_1$ and $E_2$ }} & {\scriptsize
	\makecell{(II$'$) crossings\\ of switching edges \\to paths $P_1,P_2$ }} & {\scriptsize
	\makecell{(III$'$) crossings of \\ blue edges to paths \\$Q_1, Q_1^\prime, Q_2, Q_2^\prime$ }} & total\\ 
\hline 

\{\} & \{\} & \{1,1,2,2\} & 0 & 12 & 0 & 12\\ 
\hline 
\{1\} & \{\} & \{1,2,2\} & 0 & 10 & 2 & 12\\ 
\hline 
\{1\} & \{1\} & \{2,2\} & 0 & 8 & 4 & 12\\ 
\hline 
\{2\} & \{\} & \{1,1,2\} & 0 & 8 & 4 & 12\\ 
\hline 
\{2\} & \{1\} & \{1,2\} & 0 & 6 & 6 & 12\\ 
\hline 
\{2\} & \{2\} & \{1,1\} & 0 & 4 & 8 & 12\\ 
\hline 
\{2\} & \{1,1\} & \{2\} & 1 & 4 & 8 & 13\\ 
\hline 
\{1,1\} & \{\} & \{2,2\} & 1 & 8 & 4 & 13\\ 
\hline 
\{1,2\} & \{\} & \{1,2\} & 2 & 6 & 6 & 14\\ 
\hline 
\{1,2\} & \{1\} & \{2\} & 2 & 4 & 8 & 14\\ 
\hline 
\{1,2\} & \{2\} & \{1\} & 2 & 2 & 10 & 14\\ 
\hline 
\{1,2\} & \{1,2\} & \{\} & 2+2 & 0 & 12 & 16\\ 
\hline 
\{2,2\} & \{\} & \{1,1\} & 4 & 4 & 4 & 12\\ 
\hline 
\{2,2\} & \{1\} & \{1\} & 4 & 2 & 6 & 12\\ 
\hline 
\{2,2\} & \{1,1\} & \{\} & 4+1 & 0 & 8 & 13\\ 
\hline
\{1,1,2\} & \{\} & \{2\} & 5 & 4 & 4 & 13\\ 
\hline 
\{1,1,2\} & \{2\} & \{\} & 5 & 0 & 8 & 13\\ 
\hline 
\{1,2,2\} & \{\} & \{1\} & 8 & 2 & 2 & 12\\ 
\hline 
\{1,2,2\} & \{1\} & \{\} & 8 & 0 & 4 & 12\\ 
\hline 
\{1,1,2,2\} & \{\} & \{\} & 13 & 0 & 0 & 13\\ 
\hline 
\end{tabular}
\medskip
\caption{The complete case analysis for the proof of Lemma~\ref{lem:redgraynonplanar}.
	The sets $E_1$ and $E_2$ of top/bottom edges, and the set $E_s$ of switching edges are listed only
	by their weights since only the weights really matter in the claims (I$'$), (II$'$) and (III$'$).
	The first column is ordered by the increasing combined weight.
	Columns $4$ to $6$ list lower bounds on the exclusive contributions of each of the claims (I$'$) to (III$'$)
	to the total number of refined crossings.}
\label{tbl2}
\end{table}

When checking each row of the lower-bound entries in Table~\ref{tbl2}, note that the claim (I$'$) separately
contributes all mutual crossings within the set $E_1$ and all within $E_2$.
The claim (II$'$) contributes at least $2\cdot|E_s|$ exclusive refined crossings by the switching blue edges to the paths $P_1, P_2$.
The claim (III$'$) contributes in total at least $\min_{p+q=|E_s|} \big(2\min(|E_1|+p,|E_2|+q)+2\min(|E_1|+q,|E_2|+p)\big)$
exclusive refined crossings by the blue edges to the paths $Q_1, Q_1^\prime, Q_2,Q_2^\prime$.

For a more detailed illustration, we pick one of the less-trivial rows of Table~\ref{tbl2} -- row 14 with $E_1=\{2,2\}$, $E_2=\{1\}$ and $E_s=\{1\}$.
Then the two weight-$2$ top blue edges of $E_1$ mutually cross by (I$'$), contributing $4$ crossings to the sum.
The single switching edge of $E_s$ contributes $2$ exclusive crossings by (II$'$).
Regarding the contribution of (III$'$), the combined weight of the blue edges attached to $u_1,u_2,u_3,u_4$ from below is $|E_2|+b=b+1$ 
for $b\in\{0,1\}$, depending on the attachment of the single weight-$1$ switching edge $f\in E_s$, 
and then the weight of the blue edges attached to $v_1,v_2,v_3,v_4$ from below is $|E_2|+(|E_s|-b)=2-b$.
We hence get at least $2\min(b+1,6-b-1)=2b+2$ exclusive crossings of the blue edges with $Q_1,Q_2$,
and at least $2\min(2-b,6-2+b)=4-2b$ exclusive crossings of the blue edges with $Q_1',Q_2'$.
These two together give the minimum lower bound of $6$ by (III$'$).
All remaining rows of Table~\ref{tbl2} are analogous (and often not using the full scale of the outlined arguments).
\qed\end{proof}

%


We can finish our self-contained computer-free proof of $\crg(G_{13})\geq13$:

\begin{proof}[{\it of Theorem~\ref{thm:main13alt}}]
Let $D$ be an optimal drawing of $G_{13}$ satisfying the conclusion of Lemma~\ref{lem:redbluecr}.
Thanks to Lemma~\ref{lem:redgrayplanar}, we may assume that the subdrawing of $G_0$ within $D$ is not planar.
If the red edges of $G_0$ are not crossed, we have a crossing of two gray
edges of $G_0$ in $D$ and $12$ more refined crossings by Lemma~\ref{lem:redgraynonplanar}, altogether~$13$ crossings.

We now consider that some red edge is crossed in a point~$x$ by a gray or green edge $g$ 
(since $g$ is not blue by Lemma~\ref{lem:redbluecr}).
Then we redraw $g$ as follows:
Up to symmetry, let $x$ be closer (or equal) to $v_5$ than to~$u_5$ on~$R$.
We cut the drawing of $g$ at $x$ and route both parts of $g$ closely along
their side of the drawing of $R$, until we rejoin them at the end~$v_5$.
This redrawing~$D'$~saves $\ell\in\{3,4,5,7\}$ crossings of $g$ at~$x$,
and creates at most $\ell-1$ new refined crossings only between $g$ and the blue
edges ending on $R$ between $x$ and $v_5$, as one can check in Fig.~\ref{fig:main13}.
So, the number of refined crossings in $D'$ is at most $\crg(G_{13})-1$.
We possibly repeat the same procedure for all other crossings of red edges in~$D'$, and denote by $D''$ the resulting drawing.

We apply Lemma~\ref{lem:redgraynonplanar} to $D''$, finding at least $12$ refined crossings.
Since the number of refined crossings in $D''$ is at most the same number in $D'$,
and that is in turn at most $\crg(G_{13})-1$, we obtained desired $\crg(G_{13})\geq12+1=13$.
\qed\end{proof}


\section{The Improved Construction for Theorem~\ref{thm:main13impro}}\label{sec:improved}

\subsection{A generalized $13$-crossing-critical family}

We give the generalized definition, extending the family of Theorem~\ref{thm:main13}(b), in two stages.

\begin{definition}[Graphs $G_{13}^k$]\label{def:ccg13k}
Let $k\geq2$ be an integer.
Let $G_0$ be the induced subgraph of the graph $G_{13}$ from Figure~\ref{fig:main13}(a) on the vertex set
$\{u_5,u_4,u_3,u_2,u_1,x_1,x_2,v_1,v_2,v_3,v_4,v_5\}$.
Let $Q$ denote the path on $2k$ vertices $x^1,y^1,x^2,y^2,\ldots,x^k,y^k$ in this order,
and with all edges as multiedges of weight~$7$.
Let $G_1$ be the graph on $10+2k$ vertices obtained from the graph $G_0-x_1x_2$ by identifying $x_1$ with $x^1$ and $x_2$ with~$y^k$.

Let $D_i$, for $i\in\{1,\ldots,k\}$, denote the graph on the vertex set $\{x^i,y^i,w_1^i,w_2^i,w_3^i,w_4^i\}$
(where $x^i,y^i$ are from~$Q$) with the edges $x^iw_1^i$, $y^iw_4^i$, $w_1^iw_4^i$, $w_2^iw_3^i$ of weight $1$
and the edges $w_1^iw_2^i$ and $w_3^iw_4^i$ of weight~$2$.
From the union $G_1\cup D_1\cup\ldots\cup D_k$ we obtain the graph $G_{13}^k$ via identifying $u_5$ with $w_2^1$ and $w_3^k$ with $v_5$,
and for $i=2,3,\ldots,k$, identifying $w_3^{i-1}$ with $w_2^{i}$.
\end{definition}

\begin{definition}[Graphs $G_{13}^{(k_1,\ldots,k_m)}$]\label{def:ccg13kkk}
Let $m\geq1$ and $k_1,\ldots,k_m$ be positive half-integers%
\footnote{A half-integer is an integer multiple of $\frac12$, and we choose these in the definition in order to stay compatible
with the original definition in \cite{DBLP:journals/combinatorica/BokalDHLMW22}, where $k_i$ were just integers.},
such that~$k_1+\ldots+k_m=k$ is an integer and that $k_1+\ldots+k_j$ is also an integer for some~$1\leq j<m$.%
\footnote{See the proof of Lemma~\ref{lem:newgeq13} for use of this auxiliary condition.}
Consider the graph $G_{13}^k$ and the subpath $Q\subseteq G_{13}^k$ from Definition~\ref{def:ccg13k},
and let $Q_1,\ldots,Q_m\subseteq Q$ be pairwise disjoint subpaths of $Q$, consecutive in this order, such that $|V(Q_i)|=2k_i$ for~$i=1,\ldots,m$.
The graph $G_{13}^{(k_1,\ldots,k_m)}$ is obtained from $G_{13}^k$ by contracting each of the paths $Q_1,\ldots,Q_m$ into one vertex.
\end{definition}

For an illustration, the graph $G_{13}$ of Figure~\ref{fig:main13}(a) is isomorphic to $G_{13}^{(1,1)}$,
and further Figure~\ref{fig:splitwedge} shows the graph $G_{13}^{(1\!/2,\,1\!/2,\,1)}$.
The (sub)graph $D_i$ of $G_{13}^k$ is called the $i$-th {\em wedge}, and this notation is naturally extended to $G_{13}^{(k_1,\ldots,k_m)}$.
We call the contracted path $Q$ in the previous definition the {\em spine of}~$G_{13}^{(k_1,\ldots,k_m)}$.
Analogously to the previous section, we say that the edges of $Q$ are {\em red}, and the edges between
vertices $u_i$ and $v_j$ of $Q$ are {\em blue}.

Our first goal is to prove that all these graphs $G_{13}^{(k_1,\ldots,k_m)}$ require at least $13$ crossings.
For that we use also the following two claims from previous related research which are, in particular the first one,
now ``computer-free'' thanks to Theorem~\ref{thm:main13alt}.
\begin{lemma}[{\cite[Theorem~5.7]{DBLP:journals/combinatorica/BokalDHLMW22}}]\label{lem:13integ}\hfill
Let $m\geq2$ and $k_1,\ldots,k_m$ be positive integers.
Then $\crg\big(G_{13}^{(k_1,\ldots,k_m)}\big)\geq13$.
\end{lemma}

\begin{lemma}[{\rm from the proof of} {\cite[Lemma~5.6]{DBLP:journals/combinatorica/BokalDHLMW22}}]\label{lem:wedge-}
Let $m\geq j\geq1$ and $k_1,\ldots,k_m$ be positive half-integers, such that $k_j\geq\frac32$.
Then $\crg\big(G_{13}^{(k_1,\ldots,k_j,\ldots,k_m)}\big)\geq\crg\big(G_{13}^{(k_1,\ldots,k_j-1,\ldots,k_m)}\big)$.
\end{lemma}
Since Lemma~\ref{lem:wedge-} is only implicitly contained in the proof of mentioned \cite[Lemma~5.6]{DBLP:journals/combinatorica/BokalDHLMW22}
(which has a more restrictive setup), we also provide a self-contained proof of it in the Appendix.

We now finish the lower bound proof:
\begin{lemma}\label{lem:newgeq13}
Let $m\geq3$ and $k_1,\ldots,k_m$ be positive half-integers such that $k_1\geq1$ and $k_m\geq1$.
Then $\crg\big(G_{13}^{(k_1,\ldots,k_m)}\big)\geq13$.
\end{lemma}

\begin{proof}
Let $k_1+\ldots+k_m=k$.
Consider an optimal drawing of $G:=G_{13}^{(k_1,\ldots,k_m)}$, i.e., one with $\crg(G)$ crossings.
Let $Q$ be the spine of $G$.
Since every edge of $Q$ is of weight $7$, at most one edge $f$ of $Q$ is crossed, or the lemma is already true.
Assume first that $f\not=x^iy^i$ for any $i\in\{1,\ldots,m\}$ (i.e., $f$ does not belong to any of the wedges of~$G$).
Then we contract all edges of $Q$ except one, which is preferably~$f$, such that we get a graph isomorphic to
$G^{(l_1,l_2)}_{13}$ for some integers $l_1+l_2=k$.
Since contractions of uncrossed edges do not raise the crossing number, we have that 
$\crg(G)\geq\crg(G^{(l_1,l_2)}_{13})\geq13$ by Lemma~\ref{lem:13integ}.

Assume now the opposite case (in which the previous argument with integers $l_1+l_2=k$ would fail); 
that $f=x^iy^i$ where $i\in\{1,\ldots,m\}$.
Let $1\leq j<m$ be such that $k_1+\ldots+k_j=k'$ is an integer, as assumed in Definition~\ref{def:ccg13kkk}.
Then $f'=y^{k'}x^{k'+1}$ is an edge of $Q$ not belonging to a wedge.
We contract all edges of $Q$ except $f$ and~$f'$, and we get (up to symmetry) a graph isomorphic to
$G^{(l_1,l_2,l_3)}_{13}$ for some $l_1+l_2+l_3=k$ such that $l_3$ is an integer and $l_1,l_2$ are odd multiples of~$\frac12$. 
Again, $\crg(G)\geq\crg(G^{(l_1,l_2,l_3)}_{13})$ since only uncrossed edges have been contracted.
By an iterative application of Lemma~\ref{lem:wedge-}, we arrive at the conclusion that
\mbox{$\crg(G^{(l_1,l_2,l_3)}_{13})\geq \crg(G_{13}^{(1\!/2,\,1\!/2,\,1)})$.}

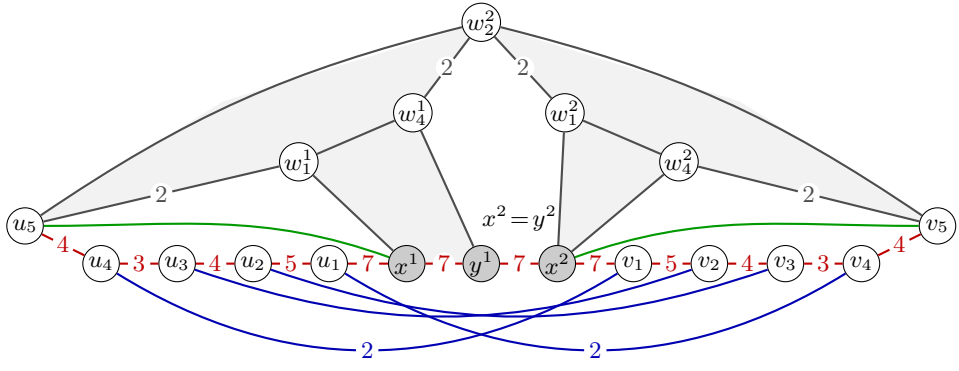
\begin{figure}[t]
 \centering
 \begin{tikzpicture}[scale=1.1]\footnotesize
  \tikzset{every node/.style={labeled}}
  \tikzset{every path/.style={thick}}
  \node[fill=black!20!white] (x) at (-0.98,0) {$x^1$};
  \node[fill=black!20!white] (y) at (0,0) {$y^1$};
  \node[fill=black!20!white] (xx) at (1,0) {$x^2$};
  \node[draw=none] at (0.5,0.6) {\footnotesize$x^2\!=\!y^2$};
  \node (u1) at (-2,0) {$u_1$};
  \node (u2) at (-3,0) {$u_2$};
  \node (u3) at (-4,0) {$u_3$};
  \node (u4) at (-5,0) {$u_4$};
  \node (u5) at (-6,0.5) {$u_5$};
  \node (v1) at (2,0) {$v_1$};
  \node (v2) at (3,0) {$v_2$};
  \node (v3) at (4,0) {$v_3$};
  \node (v4) at (5,0) {$v_4$};
  \node (v5) at (6,0.5) {$v_5$};
  \node (h2) at (-2.4,1.35) {$w^{1}_1$};
  \node (h3) at (-0.9,2) {$w^{1}_4$};
  \node (k1) at (0,3.2) {$w^2_2$};
  \node (k2) at (1.1,2) {$w^{2}_1$};
  \node (k3) at (2.6,1.35) {$w^{2}_4$};
  \tikzset{every node/.style={edge-label}}
  \draw[green!60!black] (x) to[in=0, out=160] (u5);
  \draw[green!60!black] (xx) to[out=20, in=180] (v5);
  \draw[red!75!black] (u5) -- node {$4$} (u4) -- node {$3$} (u3)
   -- node {$4$} (u2) -- node {$5$} (u1) -- node {$7$} (x) -- node {$7$} (y) -- node {$7$} (xx)
   -- node {$7$} (v1) -- node {$5$} (v2) -- node {$4$} (v3) -- node {$3$} (v4) -- node {$4$} (v5);
  \draw[blue!70!black]
   (u2) edge[bend right=19] (v3)
   (u3) edge[bend right=19] (v2)
   (u1) edge[bend right=32] node {2} (v4)
   (u4) edge[bend right=32] node {2} (v1)
  ;
  \draw[black!70!white] (u5) edge[thick] node {2} (h2) (h2) -- (h3) (h3) edge[thick] node {2} (k1)
     (k1) edge[bend right=10] (u5)
   (k1) edge[thick] node {2} (k2) (k2) -- (k3) (k3) edge[thick] node {2} (v5) (v5) edge[bend right=10] (k1)
   (h2) -- (x)  (y) -- (h3)  (k2) -- (xx) -- (k3)
  ;
  \begin{scope}[on background layer]
    \fill[color=black!5!white] (y.center) -- (x.center) -- (h2.center) -- (u5.center)
      -- (-3.4,2.15) -- (k1.center) -- (h3.center) -- cycle;
    \fill[color=black!5!white] (xx.center) -- (k3.center) -- (v5.center)
      -- (3.4,2.15) -- (k1.center) -- (k2.center) -- cycle;
  \end{scope}
 \end{tikzpicture}

 \caption{The graph~$G_{13}^{(1\!/2,\,1\!/2,\,1)}$.}
 \label{fig:splitwedge}
\end{figure}

In the rest, we analyze the number of crossing in an optimal drawing of the graph 
$G':=G_{13}^{(1\!/2,\,1\!/2,\,1)}$ which is depicted in Figure~\ref{fig:splitwedge}.
We may also assume that in every optimal drawing of~$G'$, the edge $f=x^1y^1$ is crossed, and so exactly once,
or we are readily finished by a contraction down to $G_{13}^{(1,1)}=G_{13}$.
By Proposition~\ref{pro:folklorecr}, $f$ is not crossed by $x^1w_1^1$, $y^1w_4^1$ or~$x^1u_5$,
and $f$ is not crossed by any edge of weight greater than $1$ which would already give $14$ crossings.

Assume that $f$ is crossed by the edge $f_1=w_1^1w_4^1$.
We denote by $D_1,\ldots,D_6$ the six pairwise edge-disjoint cycles in $G'$ formed by each one of the six
blue edges (considered as multiple) and the corresponding subpaths of $Q$; see in Figure~\ref{fig:splitwedge}.
All these cycles contain one edge of the multiedge $x^1y^1$.
Consider moreover the cycle $C_1=(w_1^1,w_4^1,w_2^2,u_5)$ where $f_1\in E(C_1)$.
Then $C_1$ transversely crosses each cycle $D_i$, $i\in\{1,\ldots,6\}$, and since $C_1$ and $D_i$ are
are vertex disjoint, they have another crossing which is distinct from the other ones considered.
Altogether, we see at least $7+6=13$ crossings, as desired.
The same argument covers also the cases of $f$ being crossed by one of the edges $u_5w_2^2$, $w_1^2w_4^2$ and $w_2^2v_5$.

Assume that $f$ is crossed by an edge $f_2\in\{x^2w_1^2,x^2w_4^2,x^2v_5\}$.
Since there is no other crossing on the edges $x^1y^1$ and $y^1x^2$, we may simply redraw $f_2$ closely along the
path $(x^1,y^1,x^2)$ towards its end $x^2$ (possibly crossing $y^1w_4^1$) instead of crossing~$f$, 
which contradicts optimality of the drawing.
It remains to consider that $f$ is crossed by a (blue) edge $f_3\in\{u_3v_2,u_2v_3\}$.
Observe that there is at most one edge other than $f_3$ which crosses the spine $Q$, or we already have $7+2\cdot3=13$ crossings.
In this case we redraw $f_3$ closely along the subpath $Q_3\subseteq Q$ between the ends of $f_3$, 
for which we have two possibilities (one on each ``side'' of~$Q_3$).
Since there are altogether $6+5=11$ edges incident to the internal vertices of $Q_3$ (counting with weights),
one of the two possibilities of redrawing $f_3$ along $Q_3$ crosses at most $5$ of those edges, and it possibly
crosses also the at most one other edge crossing $Q$.
We have got at most $6$ crossings on $f_3$, which again contradicts optimality of the drawing.
We are done.
\qed\end{proof}

The last part in this subsection deals with crossing-criticality of the above defined family.

\begin{theorem}\label{thm:13altcrit}
Let $m\geq3$ and $k_1,\ldots,k_m$ be positive half-integers such that $k_1\geq1$ and $k_m\geq1$.
Then $G_{13}^{(k_1,\ldots,k_m)}$ is a $13$-crossing-critical graph.
\end{theorem}

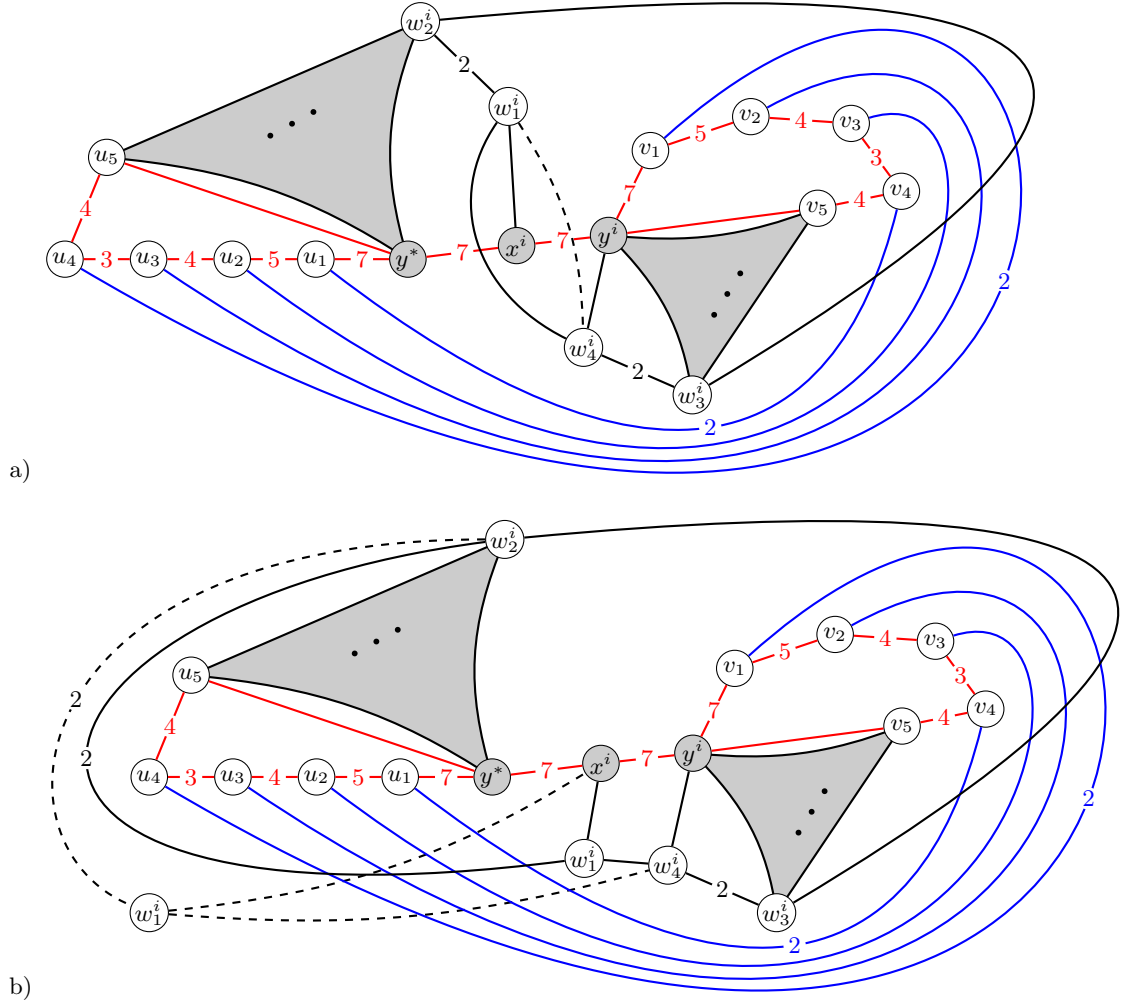
\begin{figure}[t]
 \centering
a)
 \begin{tikzpicture}[xscale=1.1, yscale=0.9]\footnotesize
   \path[use as bounding box] (-5.3,-3.2) rectangle (7.6,4);
  \tikzset{every node/.style={labeled}}
  \tikzset{every path/.style={thick}}
  \node[fill=black!20!white] (x) at (-0.9,0) {$y^{*}$};
  \node[fill=black!20!white] (y) at (0.4,0.2) {$x^{i}$};
  \node[fill=black!20!white] (xx) at (1.5,0.35) {$y^i$};
  \node (u1) at (-2,0) {$u_1$};
  \node (u2) at (-3,0) {$u_2$};
  \node (u3) at (-4,0) {$u_3$};
  \node (u4) at (-5,0) {$u_4$};
  \node (u5) at (-4.5,1.5) {$u_5$};
  \node (v1) at (2,1.6) {$v_1$};
  \node (v2) at (3.2,2.1) {$v_2$};
  \node (v3) at (4.4,2) {$v_3$};
  \node (v4) at (5,1) {$v_4$};
  \node (v5) at (4,0.75) {$v_5$};
  \node (h1) at (-0.75,3.5) {$w^{i}_2$};
  \node (h2) at (0.3,2.25) {$w^{i}_1$};
  \node (h3) at (1.2,-1.3) {$w^{i}_4$};
  \node (k1) at (2.5,-2) {$w^i_3$};
  \tikzset{every node/.style={edge-label}}
  \draw[red] (x) -- (u5) -- node {$4$} (u4) -- node {$3$} (u3) -- node {$4$}
   (u2) -- node {$5$} (u1) -- node {$7$} (x) -- node {$7$} (y) -- node {$7$} (xx) -- node {$7$} (v1) -- node {$5$}
   (v2) -- node {$4$} (v3) -- node {$3$} (v4) -- node {$4$} (v5) -- (xx);
  \draw[blue]
   (u2) .. controls (5,-7.5) and (7,3) .. (v3)
   (u3) .. controls (6.5,-8.8) and (8.2,5.5) .. (v2)
   (u1) .. controls (2,-3.8) and (4.2,-3) .. node {2} (v4)
   (u4) .. controls (9,-10.2) and (8.5,8.8) .. node {2} (v1)
  ;
  \draw (h1) edge[thick] node {2} (h2) (h2) to[bend right=44] (h3) (h3) edge[thick] node {2} (k1)
	 (k1) .. controls (9.3,3) and (7,4.4) .. (h1);
  \draw[dashed,thick] (h2) to[bend left=13] (h3);
  \draw (h2) -- (y) (xx) -- (h3);
  \begin{scope}[on background layer]
   \draw[fill=black!20!white] (x.center) to[bend right=16] (u5.center) -- (h1.center) to[bend right=16] cycle;
   \draw[fill=black!20!white] (xx.center) to[bend right=-16] (k1.center) --
(v5.center) to[bend right=-16] cycle;
  \end{scope}
  \node[rotate=29, draw=none, fill=none] (label) at (138:3) {\huge$\cdots$};
  \node[rotate=58, draw=none, fill=none] (label) at (-9:3) {\huge$\cdots$};
 \end{tikzpicture}
\\[2ex]
b)
 \begin{tikzpicture}[xscale=1.1, yscale=0.9]
   \path[use as bounding box] (-6.3,-3.2) rectangle (6.6,4);
  \tikzset{every node/.style={labeled}}
  \tikzset{every path/.style={thick}}
  \node[fill=black!20!white] (x) at (-0.9,0) {$y^{*}$};
  \node[fill=black!20!white] (y) at (0.4,0.2) {$x^{i}$};
  \node[fill=black!20!white] (xx) at (1.5,0.35) {$y^i$};
  \node (u1) at (-2,0) {$u_1$};
  \node (u2) at (-3,0) {$u_2$};
  \node (u3) at (-4,0) {$u_3$};
  \node (u4) at (-5,0) {$u_4$};
  \node (u5) at (-4.5,1.5) {$u_5$};
  \node (v1) at (2,1.6) {$v_1$};
  \node (v2) at (3.2,2.1) {$v_2$};
  \node (v3) at (4.4,2) {$v_3$};
  \node (v4) at (5,1) {$v_4$};
  \node (v5) at (4,0.75) {$v_5$};
  \node (h1) at (-0.75,3.5) {$w^{i}_2$};
  \node (h2) at (0.2,-1.2) {$w^{i}_1$};
  \node (h2b) at (-5,-2) {$w^{i}_1$};
  \node (h3) at (1.2,-1.3) {$w^{i}_4$};
  \node (k1) at (2.5,-2) {$w^i_3$};
  \tikzset{every node/.style={edge-label}}
  \draw[red] (x) -- (u5) -- node {$4$} (u4) -- node {$3$} (u3) -- node {$4$}
   (u2) -- node {$5$} (u1) -- node {$7$} (x) -- node {$7$} (y) -- node {$7$} (xx) -- node {$7$} (v1) -- node {$5$}
   (v2) -- node {$4$} (v3) -- node {$3$} (v4) -- node {$4$} (v5) -- (xx);
  \draw[dashed,thick] (h2b) to[bend right=12] (h3) (h2b) to[bend right=15] (y);
  \draw[blue]
   (u2) .. controls (5,-7.5) and (7,3) .. (v3)
   (u3) .. controls (6.5,-8.8) and (8.2,5.5) .. (v2)
   (u1) .. controls (2,-3.8) and (4.2,-3) .. node {2} (v4)
   (u4) .. controls (9,-10.2) and (8.5,8.8) .. node {2} (v1)
  ;
  \draw (h1) .. controls (-7,2.5) and (-8,-2.5) .. node {2} (h2) 
	 (h2) to[bend right=0] (h3) (h3) edge[thick] node {2} (k1)
	 (k1) .. controls (9.3,3) and (7,4.4) .. (h1);
  \draw[dashed,thick]
	(h1) .. controls (-6.6,3.5) and (-7,-1) .. node {2} (h2b);
  \draw (h2) -- (y) (xx) -- (h3);
  \begin{scope}[on background layer]
   \draw[fill=black!20!white] (x.center) to[bend right=16] (u5.center) -- (h1.center) to[bend right=16] cycle;
   \draw[fill=black!20!white] (xx.center) to[bend right=-16] (k1.center) --
(v5.center) to[bend right=-16] cycle;
  \end{scope}
  \node[rotate=29, draw=none, fill=none] (label) at (138:3) {\huge$\cdots$};
  \node[rotate=58, draw=none, fill=none] (label) at (-9:3) {\huge$\cdots$};
 \end{tikzpicture}

 \caption{Two schematic drawings of the graph $G$ from the proof of Theorem~\ref{thm:13altcrit}.
	The dashed lines show alternative routings of some of the edges,
	and one may straightforwardly split the gray vertices $y^*$ and/or $y^i$ in order to obtain the full drawing of~$G$ as required.
	a) A drawing with $13$ crossings which drop down to $12$ crossings after deleting any one of the edges
	$y^*x^i$ or $x^iy^i$.
	b) A drawing with $18$ crossings which drop down to $12$ crossings after deleting any one of the edges
	$w_1^iw_2^i$, $x^iw_1^i$, $w_4^iw_1^i$ or $w_2^iw_3^i$.}
 \label{fig:splitwcrit}
\end{figure}

\begin{proof}
Thanks to Lemma~\ref{lem:newgeq13}, we only have to prove that the crossing number of $G:=G_{13}^{(k_1,\ldots,k_m)}$
drops below $13$ whenever we delete any edge $e\in E(G)$.
Let $k=k_1+\ldots+k_m$.
For every edge $e$ induced on the vertex subset $\{u_5,u_4,u_3,u_2,u_1,x^1,y^k,v_1,v_2,v_3,v_4,v_5\}$,
this has already been shown in \cite[Theorem~5.7]{DBLP:journals/combinatorica/BokalDHLMW22}, since the critical drawings exhibited there,
precisely in \cite[Figure~6]{DBLP:journals/combinatorica/BokalDHLMW22},
can easily be modified to cover also graphs $G_{13}^{(k_1,\ldots,k_m)}$ with half-integral parameters.
For the sake of completeness, we repeat the respective two drawings in the appendix here.

Regarding the (red) edges of the spine which are of the form $y^{i-1}x^{i}$ (i.e., between the $(i-1)$-th and $i$-th wedges)
or the form $x^iy^i$ (i.e., belongs to the $i$-th wedge),
sought drop in the crossing number is witnessed by the drawings in Figure~\ref{fig:splitwcrit}(a).
Note that a symmetric drawing of the $i$-th wedge can witness criticality of the next spine edge $y^ix^{i+1}$, too.
Similarly, for the remaining edges of the $i$-th wedge, drop in the crossing number is witnessed by the drawings
in Figure~\ref{fig:splitwcrit}(b) and their symmetric variants.
These drawings are used the same way both for wedges with $x^i\not=y^i$ and with contracted $x^iy^i$ for wedges with $x^i=y^i$.
We have checked all edges of~$G$.
\qed\end{proof}

\subsection{Getting all possible vertex degrees}

As one can easily check from Definition~\ref{def:ccg13kkk}, in the graph $G_{13}^{(k_1,\ldots,k_m)}$ the vertex
degree corresponding to the parameter $k_i$, where $1<i<m$, is exactly $14+2k_i$ 
(and it is by one higher for $k_1$ or $k_m$ because of the green edges).
We can thus have arbitrary numbers of vertices of degrees $15$ and higher for suitably chosen parameters $k_1,\ldots,k_m$.
We can see from Figure~\ref{fig:main13} that there are also unbounded numbers of vertices of degrees $4$ and $6$.

Vertices of degree~$3$ can be introduced into the construction with help of the following lemma:

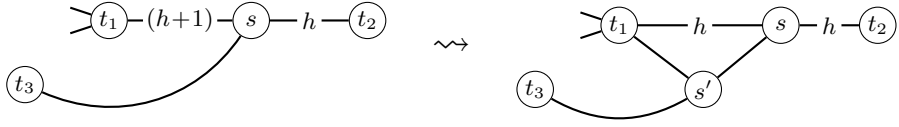
\begin{figure}[t]
 \centering\medskip
  \begin{tikzpicture}[scale=0.85]
  \tikzset{every node/.style={labeled}}
  \tikzset{every path/.style={thick}}
  \node (v2) at (1,0) {$t_1$};
  \node (v3) at (3.2,0) {$s$};
  \node (v4) at (5,0) {$t_2$};
  \node (u2) at (-0.3,-1) {$t_3$};
  \tikzset{every node/.style={edge-label}}
  \draw (v2) -- node {$(h\!+\!1)$} (v3) -- node {$h$} (v4) ;
  \draw (v2) to (0.4,0.2);
  \draw (u2) edge[bend right=40] (v3) ;
  \draw (v2) to (0.4,-0.2);
  \end{tikzpicture}
\raise7ex\hbox{\Large$\quad\leadsto\quad$}
  \begin{tikzpicture}[scale=0.85]
  \tikzset{every node/.style={labeled}}
  \tikzset{every path/.style={thick}}
  \node (v2) at (1,0) {$t_1$};
  \node (v31) at (3.5,0) {$s$};
  \node (v33) at (2.3,-1) {$s'$};
  \node (v4) at (5,0) {$t_2$};
  \node (u2) at (-0.3,-1) {$t_3$};
  \tikzset{every node/.style={edge-label}}
  \draw (v2) -- node {$h$} (v31) -- node {$h$} (v4) ;
  \draw (v2) to (0.4,0.2);
  \draw (v31) -- (v33) -- (v2) ;
  \draw (u2) edge[bend right=30] (v33) ;
  \draw (v2) to (0.4,-0.2);
  \end{tikzpicture}

 \caption{An illustration of the operation of locally introducing a vertex ($s'$) of degree $3$ from Lemma~\ref{lem:plus3deg}.
 This operation can be applied, e.g., to vertices $t_1=w_2^2$, $s=w_1^2$, $t_2=w_4^2$, and $t_3=x^2$ of Figure~\ref{fig:splitwedge}.}
 \label{fig:v3-split}
\end{figure}

\begin{lemma}[{\cite[Lemma~6.1]{DBLP:journals/combinatorica/BokalDHLMW22}}]\label{lem:plus3deg}
Assume a graph $H$ with vertices $t_1,t_2,t_3$ and $s$ such that
\begin{itemize}
\item[a)]
vertex $s$ has no other neighbours than $t_1,t_2,t_3$ in $H$,
the edge $t_1s$ is of weight $h+1$, $t_2s$ is of weight $h$, $t_3s$ is of weight $1$,
and
\item[b)]
vertex $t_1$ is of degree at most $h+5$ in $H$.
\end{itemize}
Other edges of $H$ are not important. 

Let $H'$ be created by making the edge $t_1s$ only weight $h$, deleting the edge $t_3s$,
and adding a new vertex $s'$ adjacent via three weight-$1$ edges to the vertices $t_1,t_3$ and~$s$.
See Figure~\ref{fig:v3-split}. Then $\crg(H')\geq\crg(H)$.
Furthermore, if $H$ is a $c$-crossing-critical graph and $\crg(H'-ss')<c$, then $H'$ is also $c$-crossing-critical.
\end{lemma}
This lemma applies, e.g., to vertices $t_1=w_2^i$, $s=w_1^i$, $t_2=w_4^i$ and $t_3=x^i$ of the $i$-th wedge
for any $1<i<m$ (and any subset of such wedges),
where $c=13$, $h=1$, $t_1$ is indeed of degree~$6$, and $\crg(H'-ss')<13$ is witnessed by a drawing obtained
from the one in Figure~\ref{fig:splitwcrit}(b) in which the path $(x^i,s',w_2^i)$ is routed directly without crossing.

It thus remains to cover the vertex degrees of $5,7,8,\ldots,14$.
For that we introduce a statement conceptually similar to Lemma~\ref{lem:plus3deg} for such degrees
(equal to $2b+1$ or $2b+2$ in the next).

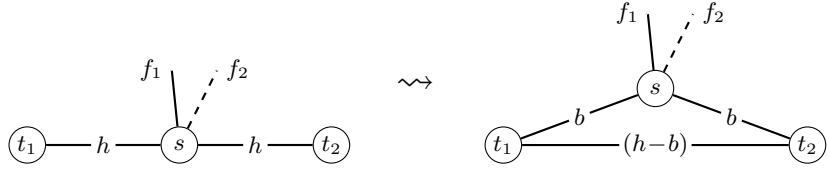
\begin{figure}[t]
 \centering\medskip
  \begin{tikzpicture}
  \tikzset{every node/.style={labeled}}
  \tikzset{every path/.style={thick}}
  \node (t1) at (1,0) {$t_1$};
  \node (s) at (3,0) {$s$};
  \node (t2) at (5,0) {$t_2$};
  \tikzset{every node/.style={edge-label}}
  \draw (t1) -- node {$h$} (s) -- node {$h$} (t2) ;
  \draw (s) -- (2.9,1) node[label=left:$f_1$] {};
  \draw[dashed] (s) -- (3.5,1) node[label=right:$f_2$] {};
  \end{tikzpicture}
\raise7ex\hbox{\Large$\quad\leadsto\quad$}
  \begin{tikzpicture}
   \path[use as bounding box] (0.7,-1) rectangle (5.3,1.2);
  \tikzset{every node/.style={labeled}}
  \tikzset{every path/.style={thick}}
  \node (t1) at (1,-0.75) {$t_1$};
  \node (s) at (3,0) {$s$};
  \node (t2) at (5,-0.75) {$t_2$};
  \tikzset{every node/.style={edge-label}}
  \draw (t1) -- node {$(h\!-\!b)$} (t2) ;
  \draw (t1) -- node {$b$} (s) -- node {$b$} (t2) ;
  \draw (s) -- (2.9,1) node[label=left:$f_1$] {};
  \draw[dashed] (s) -- (3.5,1) node[label=right:$f_2$] {};
  \end{tikzpicture}

 \caption{An illustration of the operation of locally introducing a vertex ($s'$) of degree $5,7,8,\ldots$ from Lemma~\ref{lem:plus5deg}.
	 This operation can be applied, e.g., with $h=7$ to the vertex $s=y^1$ and the edge $f_1=y^1w_4^1$ of Figure~\ref{fig:splitwedge}.}
 \label{fig:v7-split}
\end{figure}

\begin{lemma}\label{lem:plus5deg}
Assume a graph $H$ with vertices $t_1,t_2$ and $s$ such that
\begin{itemize}
\item[a)]
the edges $t_1s$ and $st_2$ are each of weight $h$, and there are only $a\in\{1,2\}$ additional edge(s) $f_1$, $f_a$ incident to $s$ -- those are of weight~$1$,
and
\item[b)]
the crossing number $c=\crg(H)$ of $H$ is at most $2h+1$.
\end{itemize}
Other edges of $H$ are not important. 

Let $H'$ be constructed from $H$ by making the edges $t_1s$ and $st_2$ of weight only $b$, where $a+1\leq b \leq h-1$,
and adding a new edge $t_1t_2$ of weight~$h-b$.
See Figure~\ref{fig:v7-split}. Then $\crg(H')\geq\crg(H)=c$.
Furthermore, if $H$ is a $c$-crossing-critical graph, and either $a=1$, or $a=2$ and for every edge $e\in E(H)\setminus\{f_1,f_2\}$
there is a drawing of $H-e$ with less than $c$ crossings in which the edges $f_1$ and $f_2$ are consecutive in
the cyclic order around $s$, then $H'$ is $c$-crossing-critical, too.
\end{lemma}

\begin{proof}
Consider an optimal drawing of $H'$.
Let the edge $t_1t_2$ be crossed by $c_0$ edges (counting their weights), and similarly let $t_1s$ and $st_2$ be crossed by $c_1$ and $c_2$ edges, respectively.
First, if $c_1+c_2\leq c_0$, then we simply remove the edge $t_1t_2$ and increase the weights of $t_1s$ and $st_2$ from $b$ to $h$,
and the resulting drawing of $H$ will have at most $\crg(H')$ crossings; hence $\crg(H')\geq\crg(H)$.

Otherwise, $c_1+c_2\geq c_0+1$, and we have $\crg(H')\geq(c_1+c_2)b+c_0(h-b)\geq c_0h+b$, but $\crg(H)\leq2h+1$.
So, we have $\crg(H')\geq\crg(H)$ or $c_0\leq1$.
In the second case, we, assuming $c_1\leq c_2$ up to symmetry, delete from $H'$ the edge $st_2$.
Then we ``detach'' the $a$ edge(s) $f_1,f_a$ from their end $s$ and prolong them along the edge $t_1s$ towards $t_1$ as their new end.
We delete $t_1s$ and increase the weight of the edge $t_1t_2$ to~$h$, 
and denote by $H''$ the new drawing which is actually the graph obtained from $H$ by contracting the multiple edge $t_1s$.
This modification results in $\crg(H'')\leq\crg(H')-c_2b+c_1a-c_1b+c_0b\leq \crg(H')+(b-c_2b)+c_1(a-b)\leq \crg(H')$ crossings.
In~$H''$, we may now easily split the vertex $t_1$ to make a new vertex~$s'$ (but no new crossings), 
such that this results in a graph isomorphic to~$H$.
Therefore, we get (again) that $\crg(H)\leq\crg(H')$.

Regarding criticality of $H'$, we can straightforwardly modify the assumed drawings of $H-e$ for $e\in E(H)$ of weight~$1$
(yes, including $e\in\{f_1,f_a\}$) with less than $c$ crossings, such that in a close neighbourhood of the vertex $s$
the edges $st_1$ and $st_2$ are ``split and reconnected'' to make a graph isomorphic to~$H'-e$, drawn with less than $c$ crossings.
In addition to that, for $e'=t_1t_2\in E(H')$ of weight~$1$, we similarly obtain a drawing of $H'-e'$ by a modification
of that of $H-e$ where $e=t_1s$.
\qed\end{proof}

We will apply Lemma~\ref{lem:plus5deg}, with $h=7$, to those vertices of the spine of the graph $G_{13}^{(k_1,\ldots,k_m)}$
which are incident to only one wedge (with one or two vertices of this wedge).
All conditions of the lemma are easily satisfied with $c=13$, except possibly, when $a=2$ and $e$ is a red (spine) edge incident to~$s$,
of the condition of having a drawing of $H-e$ with less than $c$ crossings and consecutive $f_1,f_2$ around~$s$.
In the latter situation, we will use critical drawings as in Figure~\ref{fig:splitwcrit}(a) with the $i$-th wedge
incident to $t_1$ or to $t_2$ (but not the wedge incident to $s$ which does not fit).

\smallskip
We can now finish the last result of our paper.

\begin{proof}[{\it of Theorem~\ref{thm:main13impro}}]
We assume $d\geq8$ and choose $m=1+3q(d-5)$.
For $i=1,4,7,\ldots,m-3,m$, set $k_i=1$.
For $i=2,5,8,\ldots,12q-1$, set again $k_i=k_{i+1}=1$.
For $i=12q+2,12q+5,\ldots,27q-1$, set $k_i=k_{i+1}=\frac12$.
Finally, for $\ell=15,16,\ldots,d-1,d$ and $i\in I_\ell:=\{3q(\ell-6)+2,3q(\ell-6)+5,\ldots,3q(\ell-5)-1\}$,
set $k_i=\frac12(\ell-14)$, and $k_{i+1}=1$ if $\ell$ is even and $k_{i+1}=\frac12$ if $\ell$ is odd.
This fulfills the conditions of Definition~\ref{def:ccg13kkk}.

We start the construction with the graph $G:=G_{13}^{(k_1,\ldots,k_m)}$.
Let $z_i$, for $i\in\{1,\ldots,m\}$, denote the vertex of the spine of $G$ which in Definition~\ref{def:ccg13kkk}
corresponds to the parameter $k_i$ (i.e., $z_i$ results by the contraction of $Q_i$ there).
Then, clearly, each of the $q$ vertices $z_i$ for $i\in I_\ell$ is of degree exactly $\ell\geq15$.
There are at least $q$ vertices of degree $4$ and $6$ in the wedges.
To get desired vertices of degree $8$, we apply Lemma~\ref{lem:plus5deg} with $a=2$ and $b=3$ to the $q$ vertices
$s:=z_i$ for $i=2,5,\ldots,3q-1$.
Likewise, we apply Lemma~\ref{lem:plus5deg} with $a=2$ and $b=4$ to the $q$ vertices $s:=z_i$ for 
$i=3q+2,3q+5,\ldots,6q-1$, to get desired vertices of degree $10$, and similarly for degrees $12$ and $14$.
Then we apply Lemma~\ref{lem:plus5deg} with $a=1$ and $b=2$ to the $q$ vertices $s:=z_i$ for 
$i=12q+2,12q+5,\ldots,15q-1$, giving $q$ vertices of degree~$5$, and so on with degrees $7,9,11,13$.
Finally, we apply Lemma~\ref{lem:plus3deg} to selected $q$ wedges in order to create vertices of degree~$3$.
The final outcome is a $3$-connected $13$-crossing-critical graph $G'$ with 
at least $q$ vertices of each of the degrees $3$, $4$, $\ldots$,~$d$.

If $c>13$, then we may use the same trick as used in \cite[Proof of Corollary~1.3]{DBLP:journals/combinatorica/BokalDHLMW22};
we make a so-called zip product of the graph $G'$, on some of its vertices of degree $3$, with $c-13$ copies
of the $1$-crossing-critical graph $K_{3,3}$, and the resulting will be $c$-crossing-critical and $3$-connected.
\qed\end{proof}

\section{Conclusion}

We have supplemented the previous related paper \cite{DBLP:journals/combinatorica/BokalDHLMW22} with a new computer-free
proof of the lower bound on the crossing number of the critical construction in \cite{DBLP:journals/combinatorica/BokalDHLMW22}.
Furthermore, we have given the definitive positive answer to the research question raised by Bokal
[private communication] already in 2007; 
which vertex degrees, other than the degrees~$3,4,6$ which are present in nearly all known constructions, 
can occur arbitrarily often in infinite $c$-crossing-critical families for fixed values of~$c$?
The first answer to this question of Bokal was given in papers 
\cite{DBLP:journals/combinatorics/Hlineny08,DBLP:journals/combinatorics/BokalBDH19};
claiming that any combination of degrees is achievable arbitrarily often, but at the expense of
$c$ growing with the maximum requested degree.
Our paper takes the answer to a higher level, claiming that the same is possible to achieve
already with $c=13$, which is at the same time the smallest value of $c$ for which such constructions could exist,
based on \cite{DBLP:journals/combinatorica/BokalDHLMW22}.

\bibliographystyle{plain}
\bibliography{ccc-13proof}



\newpage\appendix

\section*{\Large{\sc Appendix I:} Full proof of Lemma~\ref{lem:wedge-}}

{\def\thelemma{\ref{lem:wedge-}}
\begin{lemma}[{\rm from the proof of} {\cite[Lemma~5.6]{DBLP:journals/combinatorica/BokalDHLMW22}}]\label{lem:wedge-AP}
Let $m\geq j\geq1$ and $k_1,\ldots,k_m$ be positive half-integers, such that $k_j\geq\frac32$.
Then $\crg\big(G_{13}^{(k_1,\ldots,k_j,\ldots,k_m)}\big)\geq\crg\big(G_{13}^{(k_1,\ldots,k_j-1,\ldots,k_m)}\big)$.
\end{lemma}
}

\begin{figure}[b]
\vspace*{-8ex}%
a)\hspace*{-2ex}%
 \newcommand{\contour}[5]{\draw[fill=#2!50!white,draw=none,shift=(#1.center),xscale=.8333] (0,0) -- +(#3:#5) arc (#3:#4:#5);}
 \begin{tikzpicture}[xscale=0.9, yscale=1]
  \tikzset{every node/.style={labeled}}
  \tikzset{every path/.style={thick}}
  \node (x) at (0,0) {$x^i$};
  \node (u5) at (-3.6,2) {$u_5$};
  \node (v5) at (3.6,2) {$w^{*}_3$};
  \node (a1) at (-2.2,3.2){$\!w^{i\!-\!1}_2\!$};
  \node (a2) at (0,3.7){$w^i_2$};
  \node (a3) at (2.2,3.2){$w^{i}_3$}; 
  \node (b3) at (-1.2,2.1){$\!w^{i\!-\!1}_1\!$};
  \node (b4) at (-0.4,2.3){$\!w^{i\!-\!1}_4\!$};
  \node (b6) at (1.2,2.1){$w^{i}_4$};
  \node[simple] (b5) at (0.4,2.3){$w^{i}_1$};
  \tikzset{every node/.style={edge-label}}
 \begin{scope}[on background layer]
	\draw[blue!50!white,line width=5pt] (a1.center) -- (a2.center) -- (a3.center);
	\contour{a2}{blue}{0}{180}{.41}
	\draw[green!60!gray,line width=5pt] (b3.center) -- (b4.center) -- (a2.center)
		 -- (b5.center)  -- (b6.center);
	\contour{b4}{green!60!black}{190}{60}{.48}
	\contour{a2}{green!60!black}{-100}{-80}{.44}
	\contour{b5}{green!60!black}{120}{-10}{.44}
 \end{scope}
\draw[red] (u5) to[bend right=10] (x);
\draw (a1) -- (a2) -- (a3);
\draw (b3) -- (x) -- (b4);
\draw (b5) -- (x) -- (b6);
\draw  (a1) edge[thick] node {2} (b3);
\draw (b3) -- (b4) edge[thick] node {2} (a2) (a2) edge[thick] node {2} (b5) (b5)--(b6);
\draw (b6) edge[thick] node {2} (a3);
\begin{scope}[on background layer]
\draw[fill=black!16!white] (x.center) to[bend right=10] (u5.center) -- (a1.center) to[bend right=14] cycle;
\draw[fill=black!16!white] (x.center) to[bend right=10] (a3.center) -- (v5.center) to[bend right=10] cycle;
\end{scope}
 \node[rotate=35, draw=none, fill=none] (label) at (138:3) {\huge$\cdots$};
 \node[rotate=-35, draw=none, fill=none] (label) at (41:3) {\huge$\cdots$};
\end{tikzpicture}
\hfill
b)\hspace*{-2ex}%
 \begin{tikzpicture}[xscale=1, yscale=1]
  \tikzset{every node/.style={labeled}}
  \tikzset{every path/.style={thick}}
  \node (x) at (0,0) {$x^i$};
  \node (u5) at (-3.6,2) {$u_5$};
  \node (v5) at (3.6,2) {$w^{*}_3$};
  \node (a1) at (-2.2,3.2){$\!w^{i\!-\!1}_2\!$};
  \node (a2) at (0,3.7){$w^i_2$};
  \node (a3) at (1.2,3.3){$w^{i}_3$}; 
  \node (b3) at (-1.2,2.1){$\!w^{i\!-\!1}_1\!$};
  \node (b4) at (-0.4,2.3){$\!w^{i\!-\!1}_4\!$};
  \node (b6) at (1.8,4.6){$w^{i}_4$};
  \node[simple] (b5) at (1,4.8){$w^{i}_1$};
  \tikzset{every node/.style={edge-label}}
 \begin{scope}[on background layer]
	\draw[blue!50!white,line width=5pt] (a1.center) -- (a2.center) -- (a3.center);
	\contour{a2}{blue}{-10}{190}{.43}
	\draw[green!60!gray,line width=5pt] (b3.center) -- (b4.center) -- (a2.center)
		 -- (b5.center)  -- (b6.center);
	\contour{b4}{green!60!black}{190}{60}{.44}
	\contour{a2}{green!60!black}{-100}{-20}{.43}
	\contour{b5}{green!60!black}{-130}{-10}{.44}
 \end{scope}
\draw[red] (u5) to[bend right=10] (x);
\draw (a1) -- (a2) -- (a3);
\draw (b3) -- (x) -- (b4);
\draw (b5) .. controls (4,6.5) and (2,2) .. (x) edge[out=50,in=270] (b6);
\draw  (a1) edge[thick] node {2} (b3);
\draw (b3) -- (b4) edge[thick] node {2} (a2) (a2) edge[thick] node {2} (b5) (b5)--(b6);
\draw (b6) edge[thick] node {2} (a3);
\begin{scope}[on background layer]
\draw[fill=black!16!white] (x.center) to[bend right=10] (u5.center) -- (a1.center) to[bend right=13] cycle;
\draw[fill=black!16!white] (x.center) to[bend right=10] (a3.center) -- (v5.center) to[bend right=10] cycle;
\end{scope}
 \node[rotate=35, draw=none, fill=none] (label) at (138:3) {\huge$\cdots$};
 \node[rotate=-35, draw=none, fill=none] (label) at (41:3) {\huge$\cdots$};
\end{tikzpicture}

 \caption{Two cases of vertex $w^i_2$ of the induction step in the proof of Lemma~\ref{lem:wedge-}.
	In each of them we ``shrink'' two wedges into one by drawing 
	new edges $w_1^{i-1}w_4^{i}$ (green) and $w_2^{i-1}w_3^{i}$ (blue) along the depicted paths.
	In case (a), this introduces no new crossing,while in case (b) the new crossing between the green and the blue
	is ``paid by'' a crossing, which must have been on the $4$-cycle $(x^i,w^{i-1}_4,w^i_2,w^{i}_1)$ before.}
 \label{fig:induction}
\end{figure}
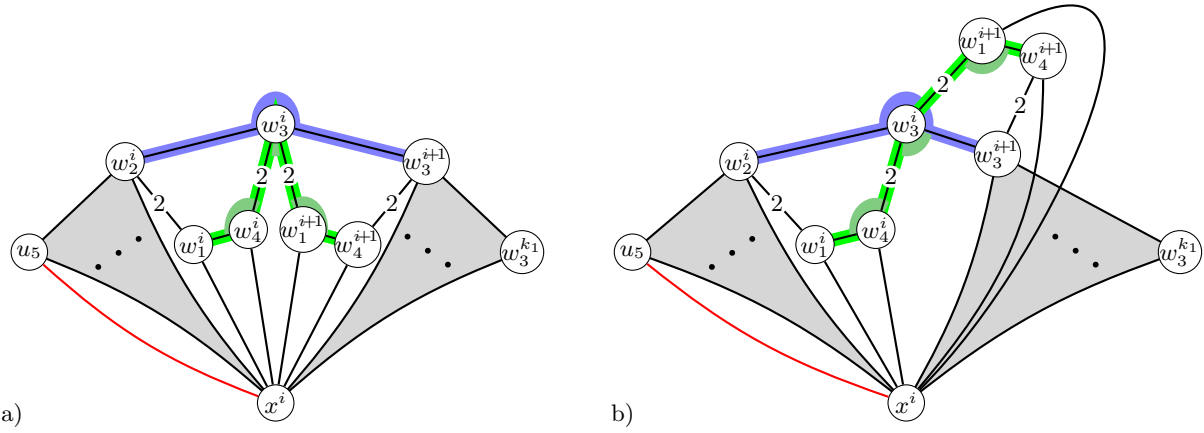

\begin{proof}
Consider an optimal drawing of the graph $G:=G_{13}^{(k_1,\ldots,k_j,\ldots,k_m)}$.
According to Definition~\ref{def:ccg13kkk}, let $i$ be such that $k_1+\ldots+k_{j-1}<i\leq k_1+\ldots+k_{j}$ and $x^i=y^i$,
which is possible since $k_j\geq\frac32$.
We may moreover assume that $y^{i-1}=x^i$; since otherwise we have $x^{i+1}=y^i$ and the subsequent arguments can be
applied symmetrically, to a ``mirror image'' of~$G$.
Our goal is to construct a drawing of $G_{13}^{(k_1,\ldots,k_j-1,\ldots,k_m)}$ with at most $\crg(G)$ crossings.

We distinguish three cases based on the cyclic order of edges leaving the vertex $w_2^i=w_3^{i-1}$:
\begin{itemize}\parskip3pt
\item
The edges incident to $w_2^i=w_3^{i-1}$, in a small neighbourhood of $w_2^i$, 
have the cyclic order $w_2^iw_4^{i-1}$, $w_2^iw_1^{i}$, $w_2^iw_3^{i}$, $w_2^iw_2^{i-1}$ (in any orientation).
See in Figure~\ref{fig:induction}\,a), where this cyclic order is counter-clockwise.
In this case, we can draw a new edge $w_1^{i-1}w_4^{i}$ along the path $(w_1^i,w_4^i,w_3^i,w_1^{i+1},w_4^{i+1})$, 
and another new edge $w_2^{i-1}w_3^{i}$ along the path $(w_2^{i-1},w_2^i,w_3^{i})$ (both new edges are of weight~$1$).
Then we delete the vertices $w_4^{i-1},w_2^i,w_1^{i}$ together with incident edges.
The resulting drawing represents a graph, which is clearly isomorphic to $G_{13}^{(k_1,\ldots,k_j-1,\ldots,k_m)}$
--- the wedges number $i-1$ and $i$ incident to~$y^{i-1}=x^i$ (and possibly to $x^{i-1}$ as well) have been replaced with one wedge.

Moreover, thanks to the assumption, we can avoid crossing between
$w_1^{i-1}w_4^{i}$ and $w_2^{i-1}w_3^{i}$ in the considered neighbourhood of former $w_2^i$.
Therefore, every crossing of the new drawing (including possible crossings
of each of the new edges $w_1^{i-1}w_4^{i}$ and $w_2^{i-1}w_3^{i}$ among themselves or with other edges)
existed already in the previous drawing of~$G$.

\item The same proof as above works if the cyclic order around $w_2^i$ is 
$w_2^iw_4^{i-1}$, $w_2^iw_1^{i}$, $w_2^iw_2^{i-1}$,~$w_2^iw_3^{i}$.

\item 
In a small neighbourhood of $w_2^i=w_3^{i-1}$, the edges have the cyclic order (in any orientation)
$w_2^iw_4^{i-1}$, $w_2^iw_3^{i}$, $w_2^iw_1^{i}$, $w_2^iw_2^{i-1}$.
See Figure~\ref{fig:induction}\,b).
Consider the $4$-cycle $C:=(x^i,w_4^{i-1},w_2^i,w_1^{i})$, which, importantly,
uses only single edges of the weight-$2$ edges incident to $w_2^i$.
In this case of the cyclic order around $w_2^i$, if $C$ is uncrossed, both sides of $C$ contain a vertex of the drawing of $G$.
Since $G-V(C)$ is connected, some edge of $C$ must be crossed in the considered drawing.
Consequently, the subdrawing of $G-E(C)$ has at most $\crg(G)-1$ crossings.

We finish similarly as in the first case, but within $G-E(C)$:
we draw a new edge $w_1^{i-1}w_4^{i}$ along the path $(w_1^{i-1},w_4^{i-1},w_2^i,w_1^{i},w_4^{i})$, 
and another new edge $w_2^{i-1}w_3^{i}$ along the path $(w_2^{i-1},w_2^i,w_3^{i})$
(both new edges are of weight~$1$, and we have so far removed only one of the two edges of each of $w_4^{i-1}w_2^i$ and $w_2^iw_1^{i}$).
These two new edges mutually cross once (at most -- in case that the named paths cross also somewhere else 
than at $w_2^i$, we may eliminate multiple crossings by standard means).
After deleting the original vertices $w_4^{i-1},w_2^i,w_1^{i}$, we hence get a drawing which is again isomorphic 
to $G_{13}^{(k_1,\ldots,k_j-1,\ldots,k_m)}$, and has at most $\crg(G-E(C))\leq\crg(G)-1+1=\crg(G)$ crossings.
\qed\end{itemize}
\end{proof}

\section*{\Large{\sc Appendix II:} Additional drawings for Theorem~\ref{thm:13altcrit}}

Here in Figure~\ref{fig:G13-3red}, we present additional two schematic drawings of the graph $G$ from the proof of Theorem~\ref{thm:13altcrit}
which come from \cite{DBLP:journals/combinatorica/BokalDHLMW22}, and which show that the crossing number of $G-e$ drops below $13$
for every edge $e$ induced on the vertex subset $\{u_5,u_4,u_3,u_2,u_1,x^1,y^k,v_1,v_2,v_3,v_4,v_5\}$.

\begin{figure}[b]
 \centering
a)
 \begin{tikzpicture}[xscale=0.9, yscale=0.90]
   \path[use as bounding box] (-6.3,-1.9) rectangle (5.7,4.8);
  \tikzset{every node/.style={labeled}}
  \tikzset{every path/.style={thick}}
  \node[fill=black!10!white] (x) at (-1,0) {$x^1$};
  \node[fill=black!10!white] (xx) at (0,0) {$y^k$};
  \node (u1) at (-2,0) {$u_1$};
  \node (u2) at (-3,0) {$u_2$};
  \node (u3) at (-4,0) {$u_3$};
  \node (u4) at (-5,0) {$u_4$};
  \node (u5) at (-4.5,1.5) {$u_5$};
  \node (v1) at (1,0) {$v_1$};
  \node (v2) at (2,0) {$v_2$};
  \node (v3) at (3,0) {$v_3$};
  \node (v4) at (4,0) {$v_4$};
  \node (v5) at (3.5,1.5) {$v_5$};
  \node (h1) at (-1,3.75) {$w^{k\!-\!1}_2$};
  \node (h2) at (-0.3,2.5) {$w^{k\!-\!1\!}_1$};
  \node (h3) at (0.75,2.44) {$w^{k\!-\!1\!}_4$};
  \node (k1) at (1.7,3.55) {$w^k_2$};
  \node (k2) at (1.6,2.1) {$w^{k}_1$};
  \node (k3) at (2.2,1.5) {$w^{k}_4$};
  \tikzset{every node/.style={edge-label}}
  \draw[red] (u5) -- node {$4$} (u4) -- node {$3$} (u3) -- node {$4$}
   (u2) -- node {$5$} (u1) -- node {$7$} (x) (xx) -- node {$7$} (v1) -- node {$5$}
   (v2) -- node {$4$} (v3) -- node {$3$} (v4) -- node {$4$} (v5);
  \draw[green!60!black] (v5) -- (xx) (x) -- (u5);
  \draw[dotted,ultra thick] (xx) -- (x);
  \draw[blue]
   (u2) edge[bend right=70] (v3)
   (u3) edge[bend right=70] (v2)
   (u1) edge[bend right=30] node {2} (v4)
   (u4) edge[dashed,bend right=30] node {2} (v1)
  ;
    \draw[blue,thick,dashed]
     (u4) .. controls (-9.2,5.8) and (10.2,5.7) .. node {2} (v1)
     (u4) .. controls (-11.7,6.75) and (15.8,5.6) .. node {2} (v1)
    ;
  \draw (h1) edge[thick] node {2} (h2) (h2) -- (h3) (h3) edge[thick] node {2} (k1) (k1) -- (h1);
  \draw (k1) edge[thick] node {2} (k2) (k2) -- (k3) (k3) edge[thick] node {2} (v5) (v5) -- (k1);
  \draw (h2) -- (xx) -- (h3);
  \draw (k2) -- (xx) -- (k3);
  \begin{scope}[on background layer]
   \draw[fill=black!20!white] (x.center) to[bend right=16] (u5.center) -- (h1.center) to[bend right=16] cycle;
  \end{scope}
  \node[rotate=33, draw=none, fill=none] (label) at (140:3) {\huge$\cdots$};
 \end{tikzpicture}
\\[1ex]b)%
 \begin{tikzpicture}[xscale=0.9, yscale=0.82]
   \path[use as bounding box] (-5.2,-1.7) rectangle (7.1,4.2);
  \tikzset{every node/.style={labeled}}
  \tikzset{every path/.style={thick}}
  \node[fill=black!10!white] (x) at (-1,0) {$x^1$};
  \node[fill=black!10!white] (xx) at (0,0) {$y^k$};
  \node (u1) at (-2,0) {$u_1$};
  \node (u2) at (-3,0) {$u_2$};
  \node (u3) at (-4,0) {$u_3$};
  \node (u4) at (-5,0) {$u_4$};
  \node (u5) at (-4.5,1.5) {$u_5$};
  \node (v1) at (1.5,1) {$v_1$};
  \node (v2) at (3.2,2.1) {$v_2$};
  \node (v3) at (4.4,2) {$v_3$};
  \node (v4) at (5,1) {$v_4$};
  \node (v5) at (3.5,0.35) {$v_5$};
  \node (h1) at (-1,3.75) {$w^{k\!-\!1}_2$};
  \node (h2) at (-0.3,2.5) {$w^{k\!-\!1\!}_1$};
  \node (h3) at (0.75,2.44) {$w^{k\!-\!1\!}_4$};
  \node (k1) at (1.7,3.55) {$w^k_2$};
  \node (k2) at (1.6,2.1) {$w^{k}_1$};
  \node (k3) at (2.4,0.8) {$w^{k}_4$};
  \tikzset{every node/.style={edge-label}}
  \draw[red] (x) -- (u5) -- node {$4$} (u4) -- node {$3$} (u3) -- node {$4$}
   (u2) -- node {$5$} (u1) -- node {$7$} (x) (xx) -- node {$7$} (v1) -- node {$5$}
   (v2) -- node {$4$} (v3) -- node {$3$} (v4) -- node {$4$} (v5) -- (xx);
  \draw[green!60!black] (v5) -- (xx) (x) -- (u5);
  \draw[dotted,ultra thick] (xx) -- (x);
  \draw[blue]
   (u2) .. controls (6,-3.5) and (7,3) .. (v3)
   (u3) .. controls (7.5,-4.5) and (8,5) .. (v2)
   (u1) edge[out=335,in=235] node {2} (v4)
   (u4) .. controls (10,-6.5) and (9.5,10.2) .. node {2} (v1)
  ;
  \draw (h1) edge[thick] node {2} (h2) (h2) -- (h3) (h3) edge[thick] node {2} (k1) (k1) -- (h1);
  \draw (k1) edge[thick] node {2} (k2) (k2) -- (k3) (k3) edge[thick] node {2} (v5) (v5) -- (k1);
  \draw (h2) -- (xx) -- (h3);
  \draw (k2) -- (xx) -- (k3);
  \begin{scope}[on background layer]
   \draw[fill=black!20!white] (x.center) to[bend right=16] (u5.center) -- (h1.center) to[bend right=16] cycle;
  \end{scope}
  \node[rotate=33, draw=none, fill=none] (label) at (140:3) {\huge$\cdots$};
 \end{tikzpicture}

 \caption{Two drawings of the graph $G$ of the proof of Theorem~\ref{thm:13altcrit}, taken from \cite[Figure~6]{DBLP:journals/combinatorica/BokalDHLMW22}.
	a) A drawing with $13$ crossings and with three alternate routings of the dashed edge.
	b) A drawing with $14$ crossings in which the vertices $v_1,v_2,v_3$ can be ``pulled'' towards $y^k$ or ``pushed''
	away from $y^k$ without changing the number of crossings.
	After deleting any one, up to symmetry, of the blue or green edges, or of the red edges except those between
	$x^1$ and $y^k$, one of the depicted drawings drops down to $12$ or less crossings.}
 \label{fig:G13-3red}
\end{figure}
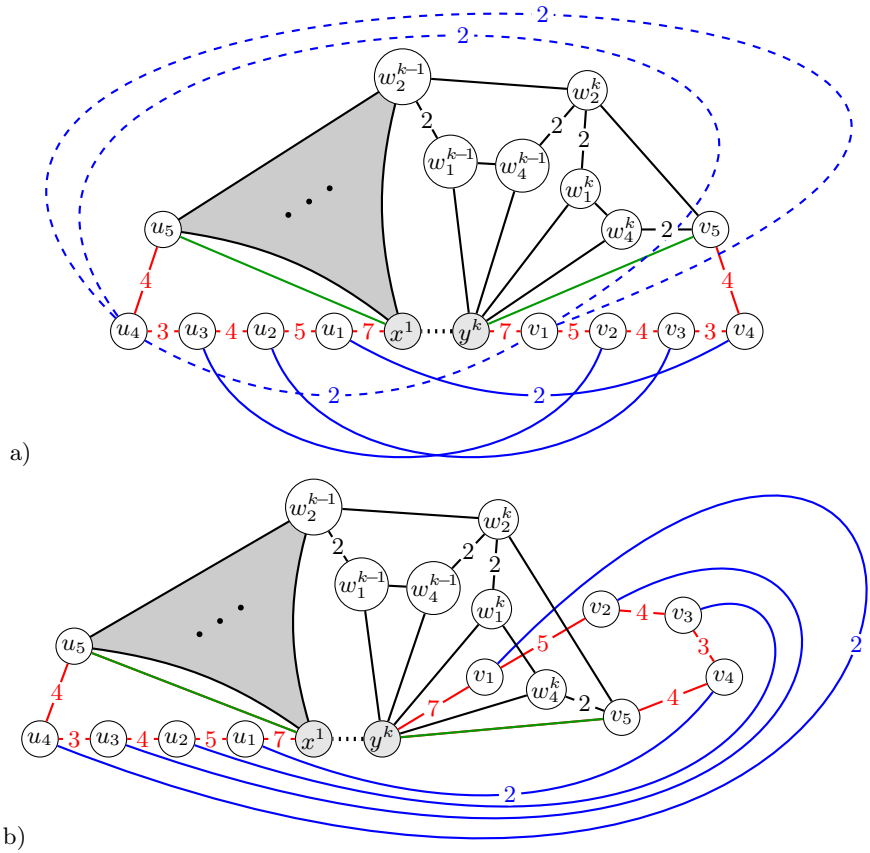

\end{document}